\documentclass[11pt]{article}%
\usepackage[T1]{fontenc}
\usepackage{textcomp}
\usepackage{cite}
\usepackage{amsmath}
\usepackage{amsfonts}
\usepackage{amssymb}
\usepackage{graphicx}
\usepackage{hyperref}%
\providecommand{\U}[1]{\protect\rule{.1in}{.1in}}

\newtheorem{theorem}{Theorem}[section]

\newtheorem{corollary}[theorem]{Corollary}

\newtheorem{definition}[theorem]{Definition}

\newtheorem{lemma}[theorem]{Lemma}

\newtheorem{proposition}[theorem]{Proposition}
\newtheorem{remark}[theorem]{Remark}

\newenvironment{proof}[1][Proof]{\noindent\textbf{#1.} }{\ \rule{0.5em}{0.5em}}
\begin{document}

\title{Asymptotic behavior as $p\rightarrow\infty$ of least energy solutions of a
$(p,q(p))$-Laplacian problem}
\author{C.O. Alves$^{\text{\thinspace a}}$, G. Ercole$^{\text{\thinspace b}}%
$\thanks{Corresponding author}\thinspace\ and G.A. Pereira$^{\text{\thinspace
b}}$\\{\small {$^{\mathrm{a}}$} Universidade Federal de Campina Grande, Campina
Grande, PB, 58.109-970, Brazil. }\\{\small E-mail: coalves@mat.ufcg.edu.br}\\{\small {$^{\mathrm{b}}$} Universidade Federal de Minas Gerais, Belo
Horizonte, MG, 30.123-970, Brazil.}\\{\small E-mails: grey@mat.ufmg.br and gilbertoapereira@yahoo.com.br}}
\maketitle

\begin{abstract}
We study the asymptotic behavior, as $p\rightarrow\infty,$ of the least energy
solutions of the problem%
\[
\left\{
\begin{array}
[c]{lll}%
-\left(  \Delta_{p}+\Delta_{q(p)}\right)  u=\lambda_{p}\left\vert
u(x_{u})\right\vert ^{p-2}u(x_{u})\delta_{x_{u}} & \mathrm{in} & \Omega\\
u=0 & \mathrm{on} & \partial\Omega,
\end{array}
\right.
\]
where $x_{u}$ is the (unique) maximum point of $\left\vert u\right\vert ,$
$\delta_{x_{u}}$ is the Dirac delta distribution supported at $x_{u},$
\[
\lim_{p\rightarrow\infty}\frac{q(p)}{p}=Q\in\left\{
\begin{array}
[c]{lll}%
(0,1) & \mathrm{if} & N<q(p)<p\\
(1,\infty) & \mathrm{if} & N<p<q(p)
\end{array}
\right.
\]
and $\lambda_{p}>0$ is such that
\[
\min\left\{  \frac{\left\Vert \nabla u\right\Vert _{\infty}}{\left\Vert
u\right\Vert _{\infty}}:0\not \equiv u\in W^{1,\infty}(\Omega)\cap
C_{0}(\overline{\Omega})\right\}  \leq\lim_{p\rightarrow\infty}(\lambda
_{p})^{\frac{1}{p}}<\infty.
\]

\end{abstract}

\noindent\textbf{2010 AMS Classification.} 35B40, 35D40, 35J20, 35J92, 46E35.


\section{Introduction}

In this paper we first study, in Section \ref{sec2}, the existence of
nonnegative least energy solutions for the Dirichlet problem
\begin{equation}
\left\{
\begin{array}
[c]{lll}%
-(\Delta_{p}+\Delta_{q})u=\lambda\left\Vert u\right\Vert _{r}^{p-r}\left\vert
u\right\vert ^{r-2}u & \mathrm{in} & \Omega\\
u=0 & \mathrm{on} & \partial\Omega,
\end{array}
\right.  \label{pqu}%
\end{equation}
where $\Omega$ is a smooth bounded domain of $\mathbb{R}^{N},$ $N\geq2,$
\[
(\Delta_{p}+\Delta_{q})u:=\operatorname{div}\left[  \left(  \left\vert \nabla
u\right\vert ^{p-2}+\left\vert \nabla u\right\vert ^{q-2}\right)  \nabla
u\right]
\]
is the $(p,q)$-Laplacian operator, $\lambda>0$ and $1\leq r<\infty.$ (In the
whole paper we denote by $\left\Vert \cdot\right\Vert _{s}$ the standard norm
of the Lebesgue space $L^{s}(\Omega),$ with $1\leq s\leq\infty$).

Our main results, inspired by the recent papers \cite{BM16} and \cite{EP16},
are presented in Sections \ref{sec3} and \ref{sec4}.

In Section \ref{sec3} we show the limit problem of (\ref{pqu}) as
$r\rightarrow\infty$ is the following%
\begin{equation}
\left\{
\begin{array}
[c]{lll}%
-(\Delta_{p}+\Delta_{q})u=\lambda\left\vert u(x_{u})\right\vert ^{p-2}%
u(x_{u})\delta_{x_{u}} & \mathrm{in} & \Omega\\
u=0 & \mathrm{on} & \partial\Omega,
\end{array}
\right.  \label{lpq}%
\end{equation}
where $x_{u}$ is the (unique) maximum point of $\left\vert u\right\vert $ and
$\delta_{x_{u}}$ is the Dirac delta distribution supported at $x_{u}.$

More precisely, we prove in Proposition \ref{existpqinf} that if
$\lambda>\lambda_{\infty}(p),$ where%
\begin{equation}
\lambda_{\infty}(p):=\min\left\{  \frac{\left\Vert \nabla u\right\Vert
_{p}^{p}}{\left\Vert u\right\Vert _{\infty}^{p}}:u\in W_{0}^{1,p}%
(\Omega)\setminus\left\{  0\right\}  \right\}  , \label{lampinf}%
\end{equation}
and $u_{n}$ denotes a nonnegative least energy solution of (\ref{pqu}) for
$r=r_{n}\rightarrow\infty,$ then there exists a subsequence of $\left\{
u_{n}\right\}  $ converging strongly in $X_{p,q}:=W_{0}^{1,\max\{p,q\}}%
(\Omega)$ to a nonnegative least energy solution of (\ref{lpq}).

Least energy solutions for (\ref{lpq}) are defined in this paper as the
minimizers of the energy functional
\[
J_{\lambda}(u)=\frac{1}{p}\left\Vert \nabla u\right\Vert _{p}^{p}+\frac{1}%
{q}\left\Vert \nabla u\right\Vert _{q}^{q}-\frac{\lambda}{p}\left\Vert
u\right\Vert _{\infty}^{p},
\]
either on $W_{0}^{1,q}(\Omega),$ if $N<p<q<\infty,$ or on the "Nehari set"%
\[
\mathcal{N}_{\lambda,\infty}:=\left\{  u\in W_{0}^{1,p}(\Omega):\left\Vert
\nabla u\right\Vert _{p}^{p}+\left\Vert \nabla u\right\Vert _{q}^{q}%
=\lambda\left\Vert u\right\Vert _{\infty}^{p}\right\}  ,
\]
if $N<q<p<\infty.$

Although not differentiable, the functional $u\mapsto\left\Vert u\right\Vert
_{\infty}^{p}$ has right Gateaux derivative at any $u\in C(\overline{\Omega
}).$ Using this fact we show in Proposition \ref{wsol} that the least energy
solutions of (\ref{lpq}) are weak solutions of this problem. It is simple to
verify (see Remark \ref{nonex}) that (\ref{lpq}) cannot have weak solutions
when $\lambda\leq\lambda_{\infty}(p).$

In Section \ref{sec4}, we consider $q=q(p),$ with%
\begin{equation}
\lim_{p\rightarrow\infty}\frac{q(p)}{p}=:Q\in\left\{
\begin{array}
[c]{lll}%
(0,1) & \mathrm{if} & N<q(p)<p\\
(1,\infty) & \mathrm{if} & N<p<q(p),
\end{array}
\right.  \label{defQ}%
\end{equation}
and fix $\Lambda\geq\Lambda_{\infty},$ where%
\begin{equation}
\Lambda_{\infty}:=\min\left\{  \frac{\left\Vert \nabla u\right\Vert _{\infty}%
}{\left\Vert u\right\Vert _{\infty}}:0\not \equiv u\in W^{1,\infty}%
(\Omega)\cap C_{0}(\overline{\Omega})\right\}  \label{Laminf}%
\end{equation}
and
\[
C_{0}(\overline{\Omega}):=\left\{  u\in C(\overline{\Omega}):u=0\,\mathrm{on}%
\,\partial\Omega\right\}  .
\]

Then, taking $\lambda_{p}>0$ satisfying%
\[
\lim_{p\rightarrow\infty}(\lambda_{p})^{\frac{1}{p}}=\Lambda\geq
\Lambda_{\infty}%
\]
we study the asymptotic behavior, as $p\rightarrow\infty,$ of the least energy
solutions $u_{p}$ of
\begin{equation}
\left\{
\begin{array}
[c]{lll}%
-(\Delta_{p}+\Delta_{q(p)})u=\lambda_{p}\left\vert u(x_{u})\right\vert
^{p-2}u(x_{u})\delta_{x_{u}} & \mathrm{in} & \Omega\\
u=0 & \mathrm{on} & \partial\Omega.
\end{array}
\right.  \label{pqp}%
\end{equation}

After deriving suitable estimates for $u_{p}$ in $W_{0}^{1,m}(\Omega),$ for
each $m>N,$ we use the compactness of the embedding $W_{0}^{1,m}%
(\Omega)\hookrightarrow C(\overline{\Omega})$ to prove that any sequence
$\left\{  u_{p_{n}}\right\}  ,$ with $p_{n}\rightarrow\infty,$ admits a
subsequence converging uniformly in $\overline{\Omega}$ to a function
$u_{\Lambda}\in W^{1,\infty}(\Omega)\cap C_{0}(\overline{\Omega}),$ which is
strictly positive in $\Omega$ and attains its (unique) maximum point at
$x_{\Lambda}\in\Omega.$

Moreover, we prove that $u_{\Lambda}$ is $\infty$-harmonic in the punctured
domain $\Omega\setminus\left\{  x_{\Lambda}\right\}  $, meaning that it
satisfies, in the viscosity sense,
\[%
\begin{array}
[c]{ccc}%
\Delta_{\infty}u_{\Lambda}=0 & \mathrm{in} & \Omega\setminus\left\{
x_{\Lambda}\right\}  ,
\end{array}
\]
where
\[
\Delta_{\infty}u:=\frac{1}{2}\nabla u\cdot\nabla\,\left\vert \nabla
u\right\vert ^{2}%
\]
denotes the $\infty$-Laplacian.

In addition, we show that if either $\Lambda=\Lambda_{\infty}$ or
$\Lambda>\Lambda_{\infty}$ and $Q\in(0,1),$ then $u_{\Lambda}$ realizes the
minimum in (\ref{Laminf}) and satisfies
\[
\left\Vert u_{\Lambda}\right\Vert _{\infty}=\frac{1}{\Lambda_{\infty}}\left(
\frac{\Lambda_{\infty}}{\Lambda}\right)  ^{\frac{1}{1-Q}}\quad\mathrm{and}%
\quad\left\Vert \nabla u_{\Lambda}\right\Vert _{\infty}=\left(  \frac
{\Lambda_{\infty}}{\Lambda}\right)  ^{\frac{1}{1-Q}}.
\]
Hence, taking into account that $\Lambda_{\infty}=(\left\Vert \rho\right\Vert
_{\infty})^{-1},$ where $\rho:\overline{\Omega}\rightarrow\lbrack0,\infty)$
denotes the distance function to the boundary $\partial\Omega,$ we conclude
that%
\[
0\leq u_{\Lambda}(x)\leq\left(  \frac{\Lambda_{\infty}}{\Lambda}\right)
^{\frac{1}{1-Q}}\rho(x),\quad\forall\,x\in\overline{\Omega}%
\]
and%
\[
\rho(x_{\Lambda})=\left\Vert \rho\right\Vert _{\infty}.
\]

These results are gathered in Theorems \ref{uinf} and \ref{diric}, and their
corollaries. In order to show how they fit into the recent literature, let us
provide a brief review on some related problems, involving exponents $p$ and
$q(p),$ with $p\rightarrow\infty.$

We start with a case involving the $p$-Laplacian operator and a simpler
dependence $q(p)=p,$ considered by Juutinen, Lindqvist, and Manfredi in
\cite{JLM}. In that paper, the authors studied the limit problem, as
$p\rightarrow\infty,$ of
\begin{equation}
\left\{
\begin{array}
[c]{lll}%
-\Delta_{p}u=\lambda_{p}(p)\left\vert u\right\vert ^{p-2}u & \mathrm{in} &
\Omega\\
u=0 & \mathrm{on} & \partial\Omega,
\end{array}
\right.  \label{eigenp}%
\end{equation}
where, according to the notation we use in this paper (see (\ref{lrm})),
\[
\lambda_{p}(p):=\min\left\{  \frac{\left\Vert \nabla u\right\Vert _{p}^{p}%
}{\left\Vert u\right\Vert _{p}^{p}}:u\in W_{0}^{1,p}(\Omega)\setminus\left\{
0\right\}  \right\}  .
\]
They first showed that,
\[
\lim_{p\rightarrow\infty}\left(  \lambda_{p}(p)\right)  ^{\frac{1}{p}}%
=\Lambda_{\infty}%
\]
and then, denoting by $u_{p}$ the positive, $L^{p}$-normalized weak solution
of (\ref{eigenp}), proved that any sequence $\left\{  u_{p_{n}}\right\}  ,$
with $p_{n}\rightarrow\infty,$ admits a subsequence converging uniformly in
$\overline{\Omega}$ to a function $u_{\infty}$ which is positive in $\Omega,$
$L^{\infty}$-normalized and solves, in the viscosity sense, the problem%
\begin{equation}
\left\{
\begin{array}
[c]{lll}%
\min\left\{  \left\vert \nabla u\right\vert -\Lambda_{\infty}u,-\Delta
_{\infty}u\right\}  =0 & \mathrm{in} & \Omega\\
u=0 & \mathrm{on} & \partial\Omega.
\end{array}
\right.  \label{infeig}%
\end{equation}

These results were independently obtained by Fukagai, Ito and Narukawa in
\cite{Fuka}, where the asymptotic behavior, as $p\rightarrow\infty$, of the
higher (variational) eigenvalues of the Dirichlet $p$-Laplacian were also
studied. Furthermore, in the recent paper \cite{SRS}, da Silva, Rossi and
Salort showed that (\ref{infeig}) has a unique (up to scalar multiplication)
maximal solution $\widehat{v}\in W^{1,\infty}(\Omega)\cap C_{0}(\overline
{\Omega})$ in the following sense: if $u$ is a nonnegative, $L^{\infty}%
$-normalized viscosity solution of (\ref{infeig}), then $u\leq\widehat{v}.$

Charro and Peral in \cite{ChaPe} ($q(p)<p$), and Charro and Parini in
\cite{ChaPa} ($q(p)>p$), studied the asymptotic behavior, as $p\rightarrow
\infty,$ of the positive weak solutions $u_{p}$ of the problem
\[
\left\{
\begin{array}
[c]{lll}%
-\Delta_{p}u=\lambda_{p}\left\vert u\right\vert ^{q(p)-2}u & \mathrm{in} &
\Omega\\
u=0 & \mathrm{on} & \partial\Omega,
\end{array}
\right.
\]
where $\lambda_{p}>0$ is such that $\lim_{p\rightarrow\infty}(\lambda
_{p})^{\frac{1}{p}}=\Lambda\in(0,\infty).$ A consequence of the results proved
in these papers is that the limit functions of the family $\left\{
u_{p}\right\}  ,$ as $p\rightarrow\infty,$ are viscosity solutions of the
problem
\[
\left\{
\begin{array}
[c]{lll}%
\min\left\{  \left\vert \nabla u\right\vert -\Lambda(u^{Q}),-\Delta_{\infty
}u\right\}  =0 & \mathrm{in} & \Omega\\
u=0 & \mathrm{on} & \partial\Omega,
\end{array}
\right.
\]
where here and in what follows $Q$ is given by (\ref{defQ}).

In \cite{ChaPa2} Charro and Parini proved that any uniform limit, as
$p\rightarrow\infty,$ of a sequence of positive weak solutions of the problem
\[
\left\{
\begin{array}
[c]{lll}%
-\Delta_{p}u=\lambda_{p}\left\vert u\right\vert ^{p-2}u+\left\vert
u\right\vert ^{q(p)-2}u & \mathrm{in} & \Omega\\
u=0 & \mathrm{on} & \partial\Omega,
\end{array}
\right.
\]
where $\lambda_{p}>0$ is such that $\lim_{p\rightarrow\infty}(\lambda
_{p})^{\frac{1}{p}}=\Lambda\in\lbrack0,\Lambda_{\infty}],$ must be a viscosity
solution of the problem
\[
\left\{
\begin{array}
[c]{lll}%
\min\left\{  \left\vert \nabla u\right\vert -\max\left\{  \Lambda
u,(u^{Q})\right\}  ,-\Delta_{\infty}u\right\}  =0 & \mathrm{in} & \Omega\\
u=0 & \mathrm{on} & \partial\Omega.
\end{array}
\right.
\]

Bocea and Mih\u{a}ilescu considered in \cite{BM16} the family $\left\{
u_{p}\right\}  $ of nonnegative least energy solutions of the problem
\[
\left\{
\begin{array}
[c]{lll}%
-(\Delta_{p}+\Delta_{q(p)})u=\lambda_{p}\left\vert u\right\vert ^{p-2}u &
\mathrm{in} & \Omega\\
u=0 & \mathrm{on} & \partial\Omega,
\end{array}
\right.
\]
where $\lambda_{p}>0$ is such that $\lim_{p\rightarrow\infty}(\lambda
_{p})^{\frac{1}{p}}=\Lambda\geq\Lambda_{\infty}.$ They proved that the uniform
limit, as $p\rightarrow\infty,$ of a sequence of $\left\{  u_{p}\right\}  $
solves, in the viscosity sense, the problem%
\[
\left\{
\begin{array}
[c]{lll}%
\min\left\{  \max\left\{  \left\vert \nabla u\right\vert ,\left\vert \nabla
u\right\vert ^{Q}\right\}  -\Lambda u,-\Delta_{\infty}u\right\}  =0 &
\mathrm{in} & \Omega\\
u=0 & \mathrm{on} & \partial\Omega.
\end{array}
\right.
\]

Ercole and Pereira, in \cite{EP16}, showed that
\[
\lim_{p\rightarrow\infty}(\lambda_{\infty}(p))^{\frac{1}{p}}=\Lambda_{\infty}%
\]
and proved that any positive minimizer $u_{p}$ in (\ref{lampinf}) has a unique
maximum point $x_{p}$ and is a weak solution of the problem%
\[
\left\{
\begin{array}
[c]{lll}%
-\Delta_{p}u=\lambda_{\infty}(p)\left\vert u(x_{p})\right\vert ^{p-2}%
u_{p}(x_{p})\delta_{x_{p}} & \mathrm{in} & \Omega\\
u=0 & \mathrm{on} & \partial\Omega,
\end{array}
\right.
\]
where $\delta_{x_{p}}$ denotes the Dirac delta distribution supported at
$x_{p}$ (note that $q(p)=p$). Furthermore, they proved that any normalized
sequence $\left\{  u_{p_{n}}/\left\Vert u_{p_{n}}\right\Vert _{\infty
}\right\}  ,$ with $p_{n}\rightarrow\infty,$ admits a subsequence converging
uniformly in $\overline{\Omega}$ to a function $w_{\infty}\in W^{1,\infty
}(\Omega)\cap C_{0}(\overline{\Omega}),$ which is positive in $\Omega$ and
assumes its maximum value $1$ at a unique point $x_{\ast}\in\Omega.$ Moreover,
$w_{\infty}$ realizes the minimum in (\ref{Laminf}) and satisfies%
\[
\left\{
\begin{array}
[c]{lll}%
\Delta_{\infty}u=0 & \mathrm{in} & \Omega\setminus\left\{  x_{\ast}\right\} \\
u=\rho/\left\Vert \rho\right\Vert _{\infty} & \mathrm{on} & \partial\left(
\Omega\setminus\left\{  x_{\ast}\right\}  \right)  =\partial\Omega\cup\left\{
x_{\ast}\right\}
\end{array}
\right.
\]
in the viscosity sense.

\section{Existence for $1\leq r<q^{\star}$ and $\lambda>\lambda_{r}%
(p)\label{sec2}$}

We recall that the embedding $W_{0}^{1,m}(\Omega)\hookrightarrow L^{r}%
(\Omega)$ is compact whenever%
\[
1\leq r<m^{\star}:=\left\{
\begin{array}
[c]{lll}%
\frac{Nm}{N-m} & \mathrm{if} & 1<m<N\\
\infty & \mathrm{if} & N\leq m.
\end{array}
\right.
\]

Thus, the Rayleigh quotient associated with this embedding assumes its minimum
value, which is positive:
\begin{equation}
0<\lambda_{r}(m):=\min\left\{  \frac{\left\Vert \nabla u\right\Vert _{m}^{m}%
}{\left\Vert u\right\Vert _{r}^{m}}:u\in W_{0}^{1,m}(\Omega)\setminus\left\{
0\right\}  \right\}  ,\quad1\leq r<m^{\star}. \label{lrm}%
\end{equation}

In this section we consider, in the Sobolev space
\[
X_{p,q}:=W_{0}^{1,\max\left\{  p,q\right\}  }(\Omega),
\]
the boundary value problem
\begin{equation}
\left\{
\begin{array}
[c]{lll}%
-(\Delta_{p}+\Delta_{q})u=\lambda\left\Vert u\right\Vert _{r}^{p-r}\left\vert
u\right\vert ^{r-2}u & \mathrm{in} & \Omega\\
u=0 & \mathrm{on} & \partial\Omega,
\end{array}
\right.  \label{Plam}%
\end{equation}
where $1\leq p,q<\infty,$ $p\not =q$ and $1\leq r<q^{\star}.$

The energy functional $I_{\lambda,r}:X_{p,q}\rightarrow\mathbb{R}$ associated
with (\ref{Plam}) is given by%
\[
I_{\lambda,r}(u):=\frac{1}{p}\left\Vert \nabla u\right\Vert _{p}^{p}+\frac
{1}{q}\left\Vert \nabla u\right\Vert _{q}^{q}-\frac{\lambda}{p}\left\Vert
u\right\Vert _{r}^{p}.
\]
It belongs to $C^{1}(X_{p,q})$ and its Gateaux derivative is expressed as
\[
\left\langle I_{\lambda,r}^{\prime}(u),v\right\rangle :=%
{\displaystyle\int_{\Omega}}
\left(  \left\vert \nabla u\right\vert ^{p-2}+\left\vert \nabla u\right\vert
^{q-2}\right)  \nabla u\cdot\nabla v\mathrm{d}x-\lambda\left\Vert u\right\Vert
_{r}^{p-r}%
{\displaystyle\int_{\Omega}}
\left\vert u\right\vert ^{r-2}uv\mathrm{d}x,\quad\forall v\in X_{p,q}.
\]

\begin{definition}
We say that $u\in X_{p,q}$ is a weak solution of \eqref{Plam} if
\[
\left\langle I_{\lambda,r}^{\prime}(u),v\right\rangle =0\quad\forall v\in
X_{p,q}.
\]

\end{definition}

We remark that a nontrivial weak solution of (\ref{Plam}) cannot exist if
$\lambda\leq\lambda_{r}(p).$ In fact, such a weak solution $u$ would satisfy
\[
\lambda\left\Vert u\right\Vert _{r}^{p}=\left\Vert \nabla u\right\Vert
_{p}^{p}+\left\Vert \nabla u\right\Vert _{q}^{q}>\left\Vert \nabla
u\right\Vert _{p}^{p}\geq\lambda_{r}(p)\left\Vert u\right\Vert _{r}^{p},
\]
so that $\left(  \lambda-\lambda_{r}(p)\right)  \left\Vert u\right\Vert
_{r}^{p}>0.$

We show in the sequel that the functional $I_{\lambda,r}$ has a global
minimizer whenever $1<p<q<\infty$. Thus, it is clear that such a minimizer is
a weak solution of (\ref{Plam}), since it must be a critical point of
$I_{\lambda,r}.$

In the case $1<q<p<\infty$ the functional $I_{\lambda,r}$ is not globally
bounded from below. In fact, if $e_{r}\in W_{0}^{1,p}(\Omega)$ is such that
\begin{equation}
\left\Vert e_{r}\right\Vert _{r}=1\quad\mathrm{and}\quad\left\Vert \nabla
e_{r}\right\Vert _{p}^{p}=\lambda_{r}(p), \label{er}%
\end{equation}
then
\[
I_{\lambda,r}(te_{r})=\frac{t^{q}}{q}\left\Vert \nabla e_{r}\right\Vert
_{q}^{q}-t^{p}\frac{\left(  \lambda-\lambda_{r}(p)\right)  }{p}\rightarrow
-\infty\quad\mathrm{as}\quad t\rightarrow\infty.
\]
However, as we will see soon, in this case the functional $I_{\lambda,r}$
assumes the minimum value on the Nehari manifold defined by
\[
\mathcal{N}_{\lambda,r}:=\left\{  u\in X_{p,q}\setminus\left\{  0\right\}
:\left\langle I_{\lambda,r}^{\prime}(u),u\right\rangle =0\right\}  =\left\{
u\in X_{p,q}\setminus\left\{  0\right\}  :\left\Vert \nabla u\right\Vert
_{p}^{p}+\left\Vert \nabla u\right\Vert _{q}^{q}=\lambda\left\Vert
u\right\Vert _{r}^{p}\right\}  .
\]

Note that if $u\in\mathcal{N}_{\lambda,r}$ then%
\begin{align*}
I_{\lambda,r}(u)  &  =\frac{1}{p}\left\Vert \nabla u\right\Vert _{p}^{p}%
+\frac{1}{q}\left\Vert \nabla u\right\Vert _{q}^{q}-\frac{\lambda}%
{p}\left\Vert u\right\Vert _{r}^{p}\\
&  =\frac{1}{p}\left\Vert \nabla u\right\Vert _{p}^{p}+\frac{1}{q}\left\Vert
\nabla u\right\Vert _{q}^{q}-\frac{1}{p}(\left\Vert \nabla u\right\Vert
_{p}^{p}+\left\Vert \nabla u\right\Vert _{q}^{q})\\
&  =\left(  \frac{1}{q}-\frac{1}{p}\right)  \left\Vert \nabla u\right\Vert
_{q}^{q}.
\end{align*}
Moreover, it follows from the identity
\[
\left\langle I_{\lambda,r}^{\prime}(tv),tv\right\rangle =t^{q}\left[
\left\Vert \nabla v\right\Vert _{q}^{q}-t^{p-q}\left(  \lambda\left\Vert
v\right\Vert _{r}^{p}-\left\Vert \nabla v\right\Vert _{p}^{p}\right)  \right]
,\quad v\in X_{p,q},\quad\,t>0,
\]
that if $v\in X_{p,q}\setminus\left\{  0\right\}  ,$ then $tv\in
\mathcal{N}_{\lambda,r}$ (for some $t>0$) if, and only if,%
\begin{equation}
\left\Vert \nabla v\right\Vert _{p}^{p}<\lambda\left\Vert v\right\Vert
_{r}^{p}\quad\mathrm{and}\quad t=\left(  \frac{\left\Vert \nabla v\right\Vert
_{q}^{q}}{\lambda\left\Vert v\right\Vert _{r}^{p}-\left\Vert \nabla
v\right\Vert _{p}^{p}}\right)  ^{\frac{1}{p-q}}. \label{vNeh}%
\end{equation}

A first consequence of this fact is that $\mathcal{N}_{\lambda,r}$ is not
empty, since
\[
\left\Vert \nabla e_{r}\right\Vert _{p}^{p}=\lambda_{r}(p)<\lambda
=\lambda\left\Vert e_{r}\right\Vert _{r}^{p}.
\]

For the sake of completeness we show now that a minimizer of $I_{\lambda,r}$
on $\mathcal{N}_{\lambda,r}$ is also a weak solution of (\ref{Plam}) whenever
$1<q<p<\infty.$

\begin{proposition}
Suppose that $1<q<p<\infty$ and that $u_{\lambda}\in\mathcal{N}_{\lambda,r}$
is such that $I_{\lambda,r}(u_{\lambda})\leq I_{\lambda,r}(v)$ for all
$v\in\mathcal{N}_{\lambda,r}.$ Then $u_{\lambda}$ is a weak solution of \eqref{Plam}.
\end{proposition}

\begin{proof}
Since $u_{\lambda}\in\mathcal{N}_{\lambda,r}$ we have $\left\Vert \nabla
u_{\lambda}\right\Vert _{p}^{p}<\left\Vert \nabla u_{\lambda}\right\Vert
_{p}^{p}+\left\Vert \nabla u_{\lambda}\right\Vert _{q}^{q}=\lambda\left\Vert
u_{\lambda}\right\Vert _{r}^{p}.$ Hence, for a fixed $v\in X_{p,q}$ we can
take $\delta>0$ such that $u_{\lambda}+sv\not \equiv 0$ and%
\[
\left\Vert \nabla(u_{\lambda}+sv)\right\Vert _{p}^{p}<\lambda\left\Vert
u_{\lambda}+sv\right\Vert _{r}^{p},\quad\forall s\in(-\delta,\delta).
\]

Let $\tau:(-\delta,\delta)\rightarrow(0,\infty)$ be the differentiable
function given by%
\[
\tau(s)=\left(  \frac{\left\Vert \nabla(u_{\lambda}+sv)\right\Vert _{q}^{q}%
}{\lambda\left\Vert u_{\lambda}+sv\right\Vert _{r}^{p}-\left\Vert
\nabla(u_{\lambda}+sv)\right\Vert _{p}^{p}}\right)  ^{\frac{1}{p-q}}.
\]
We can see from (\ref{vNeh}) that $\tau(s)(u_{\lambda}+sv)\in\mathcal{N}%
_{\lambda,r}$ for all $s\in(-\delta,\delta)$ and that $\tau(0)=1$ (since
$u_{\lambda}\in\mathcal{N}_{\lambda,r}$).

Taking into account that the differentiable function $\gamma:(-\delta
,\delta)\rightarrow\mathbb{R},$ defined by
\[
\gamma(s)=I_{\lambda,r}\left(  \tau(s)(u_{\lambda}+sv)\right)  ,
\]
attains its minimum value at $s=0,$ we have
\begin{align*}
0  &  =\gamma^{\prime}(0)\\
&  =\left\langle I_{\lambda,r}^{\prime}(u_{\lambda}),\tau^{\prime
}(0)u_{\lambda}+\tau(0)v)\right\rangle =\tau^{\prime}(0)\left\langle
I_{\lambda,r}^{\prime}(u_{\lambda}),u_{\lambda})\right\rangle +\tau
(0)\left\langle I_{\lambda,r}^{\prime}(u_{\lambda}),v)\right\rangle
=\left\langle I_{\lambda,r}^{\prime}(u_{\lambda}),v)\right\rangle .
\end{align*}

\end{proof}

\begin{definition}
We say that a function $u\in X_{p,q}$ is a least energy solution of
\eqref{Plam} if it minimizes the functional $I_{\lambda,r}$ either on
$X_{p,q}\setminus\left\{  0\right\}  $ in the case $1<p<q<\infty,$ or on
$\mathcal{N}_{\lambda,r}$ in the case $1<q<p<\infty.$
\end{definition}

Our main goal in this section is to prove that (\ref{Plam}) has at least one
nonnegative least energy solution. We assume that $1\leq r<q^{\star}$ and
$\lambda>\lambda_{r}(p).$

\begin{proposition}
Suppose that $1<p,q<\infty$ ($p\not =q$), $1\leq r<q^{\star}$ and
$\lambda>\lambda_{r}(p).$ The problem \eqref{Plam} has at least one
nonnegative least energy solution $u_{\lambda}$.
\end{proposition}

\begin{proof}
We start with the case $1<p<q<\infty,$ in which $X_{p,q}=W_{0}^{1,q}(\Omega).$
It is simple to verify that $I_{\lambda,r}$ is bounded from below and coercive.
In fact,%
\begin{align*}
I_{\lambda,r}(u)  &  =\frac{1}{p}\left\Vert \nabla u\right\Vert _{p}^{p}%
+\frac{1}{q}\left\Vert \nabla u\right\Vert _{q}^{q}-\frac{\lambda}%
{p}\left\Vert u\right\Vert _{r}^{p}\\
&  \geq\frac{1}{q}\left\Vert \nabla u\right\Vert _{q}^{q}-\frac{\lambda}%
{p}\left\Vert u\right\Vert _{r}^{p}\\
&  \geq\frac{1}{q}\left\Vert \nabla u\right\Vert _{q}^{q}-\frac{\lambda}%
{p}\left(  \lambda_{r}(q)\right)  ^{-p/q}\left\Vert \nabla u\right\Vert
_{q}^{p}=g(\left\Vert \nabla u\right\Vert _{q}),
\end{align*}
where the function $g:[0,\infty)\longrightarrow\mathbb{R},$ given by
$g(t):=\frac{1}{q}t^{q}-\frac{\lambda\left(  \lambda_{r}(q)\right)  ^{-p/q}%
}{p}t^{p},$ satisfies%
\[
-\infty<\inf\left\{  g(t):t\in\lbrack0,\infty)\right\}  \quad\mathrm{and}%
\quad\lim_{t\rightarrow\infty}g(t)=\infty.
\]

Thus, taking into account that $I_{\lambda,r}$ is also weakly sequentially
lower semi-continuous, there exists $u_{\lambda}\in X_{p,q}$ such that
\[
I_{\lambda,r}(u_{\lambda})=\min\left\{  I_{\lambda,r}(u):u\in X_{p,q}\right\}
.
\]

Noting that $I_{\lambda,r}(u_{\lambda})=I_{\lambda,r}(\left\vert u_{\lambda
}\right\vert )$ we can assume that $u_{\lambda}\geq0$ in $\Omega.$ In order to
show that $u_{\lambda}\not \equiv 0$ it is sufficient to check that
$I_{\lambda,r}$ assumes negative values in $X_{p,q}$ (note that $I_{\lambda
,r}(0)=0$). For this, by using a function $e_{r}\in C^{1}(\overline{\Omega
})\cap W_{0}^{1,p}(\Omega)\subset X_{p,q}$ satisfying (\ref{er}), we have%
\[
I_{\lambda,r}(te_{r})=\frac{t^{q}}{q}\left\Vert \nabla e_{r}\right\Vert
_{q}^{q}-t^{p}\frac{\left(  \lambda-\lambda_{r}(p)\right)  }{p}<0
\]
for all positive $t$ sufficiently small.

Now, we study the case $1<q<p<\infty$ in which $X_{p,q}=W_{0}^{1,p}(\Omega).$

Since $1\leq r<q^{\star}\leq p^{\star}$ (the latter inequality is an equality
only in the case $N\leq q<p$) we have%
\[
\left\Vert u\right\Vert _{r}^{q}\leq\frac{1}{\lambda_{r}(q)}\left\Vert \nabla
u\right\Vert _{q}^{q}\leq\frac{1}{\lambda_{r}(q)}\left(  \left\Vert \nabla
u\right\Vert _{q}^{q}+\left\Vert \nabla u\right\Vert _{p}^{p}\right)
=\frac{\lambda}{\lambda_{r}(q)}\left\Vert u\right\Vert _{r}^{p},\quad\forall
u\in\mathcal{N}_{\lambda,r},
\]
implying that
\begin{equation}
\left\Vert u\right\Vert _{r}\geq\left(  \frac{\lambda_{r}(q)}{\lambda}\right)
^{\frac{1}{p-q}}>0,\quad\forall u\in\mathcal{N}_{\lambda,r}. \label{lowr}%
\end{equation}
It follows that $I_{\lambda,r}$ restricted to $\mathcal{N}_{\lambda,r}$ is
bounded from below by a positive constant:%
\begin{align*}
I_{\lambda,r}(u)  &  =\left(  \frac{1}{q}-\frac{1}{p}\right)  \left\Vert
\nabla u\right\Vert _{q}^{q}\\
&  \geq\left(  \frac{1}{q}-\frac{1}{p}\right)  \lambda_{r}(q)\left\Vert
u\right\Vert _{r}^{q}\geq\left(  \frac{1}{q}-\frac{1}{p}\right)  \lambda
_{r}(q)\left(  \frac{\lambda_{r}(q)}{\lambda}\right)  ^{\frac{q}{p-q}}>0.
\end{align*}

Let us show that
\[
m_{\lambda}:=\inf\left\{  I_{\lambda,r}(u):u\in\mathcal{N}_{\lambda
,r}\right\}
\]
is attained in $\mathcal{N}_{\lambda,r}.$ Let $\left\{  u_{n}\right\}
\subset\mathcal{N}_{\lambda,r}$ be a minimizing sequence, that is,%
\[
I_{\lambda,r}(u_{n})=\left(  \frac{1}{q}-\frac{1}{p}\right)  \left\Vert \nabla
u_{n}\right\Vert _{q}^{q}\rightarrow m_{\lambda}.
\]
It follows that $\left\{  u_{n}\right\}  $ is bounded in $W_{0}^{1,q}(\Omega)$
and hence, taking into account that
\[
\left\Vert \nabla u_{n}\right\Vert _{p}^{p}\leq\left\Vert \nabla
u_{n}\right\Vert _{p}^{p}+\left\Vert \nabla u_{n}\right\Vert _{q}^{q}%
=\lambda\left\Vert u_{n}\right\Vert _{r}^{p}\leq\frac{\lambda}{\lambda
_{r}(q)^{p/q}}\left\Vert \nabla u_{n}\right\Vert _{q}^{p},
\]
we conclude that $\left\{  u_{n}\right\}  $ is also bounded in $W_{0}%
^{1,p}(\Omega).$ Thus, we can assume that, up to a subsequence, $\left\{
u_{n}\right\}  $ converges to a function $u_{\lambda},$ weakly in both spaces
$W_{0}^{1,p}(\Omega)$ and $W_{0}^{1,q}(\Omega),$ and strongly in $L^{r}%
(\Omega).$

It follows from (\ref{lowr}) that%
\[
\left\Vert u_{\lambda}\right\Vert _{r}=\lim_{n\rightarrow\infty}\left\Vert
u_{n}\right\Vert _{r}\geq\left(  \frac{\lambda_{r}(q)}{\lambda}\right)
^{\frac{1}{p-q}}>0,
\]
so that $u_{\lambda}\not \equiv 0.$

Moreover,%
\begin{align*}
\left\Vert \nabla u_{\lambda}\right\Vert _{p}^{p}  &  <\left\Vert \nabla
u_{\lambda}\right\Vert _{p}^{p}+\left\Vert \nabla u_{\lambda}\right\Vert
_{q}^{q}\\
&  \leq\liminf_{n\rightarrow\infty}(\left\Vert \nabla u_{n}\right\Vert
_{p}^{p}+\left\Vert \nabla u_{n}\right\Vert _{q}^{q})=\liminf_{n\rightarrow
\infty}\lambda\left\Vert u_{n}\right\Vert _{r}^{p}=\lambda\left\Vert
u_{\lambda}\right\Vert _{r}^{p}.
\end{align*}
Hence, $t_{\lambda}u_{\lambda}\in\mathcal{N}_{\lambda,r}$ where
\[
t_{\lambda}=\left(  \frac{\left\Vert \nabla u_{\lambda}\right\Vert _{q}^{q}%
}{\lambda\left\Vert u_{\lambda}\right\Vert _{r}^{p}-\left\Vert \nabla
u_{\lambda}\right\Vert _{p}^{p}}\right)  ^{\frac{1}{p-q}}\leq1.
\]

It follows that%
\begin{align*}
m_{\lambda}  &  \leq I_{\lambda,r}(t_{\lambda}u_{\lambda})\\
&  =(t_{\lambda})^{q}\left(  \frac{1}{q}-\frac{1}{p}\right)  \left\Vert \nabla
u_{\lambda}\right\Vert _{q}^{q}\\
&  \leq(t_{\lambda})^{q}\left(  \frac{1}{q}-\frac{1}{p}\right)  \liminf
_{n\rightarrow\infty}\left\Vert \nabla u_{n}\right\Vert _{q}^{q}=(t_{\lambda
})^{q}\liminf_{n\rightarrow\infty}I_{\lambda,r}(u_{n})=(t_{\lambda}%
)^{q}m_{\lambda}\leq m_{\lambda}.
\end{align*}
Consequently, $t_{\lambda}=1,$ $u_{\lambda}\in\mathcal{N}_{\lambda,r}$ and,
$I_{\lambda,r}(u_{\lambda})=m_{\lambda}.$

Since $\left\vert u_{\lambda}\right\vert \in\mathcal{N}_{\lambda,r}$ and
$I_{\lambda,r}(\left\vert u_{\lambda}\right\vert )=I_{\lambda,r}(u_{\lambda
})=m_{\lambda},$ we can assume that $u_{\lambda}$ is nonnegative.
\end{proof}

\section{The limit problem as $r\rightarrow\infty$\label{sec3}}

In this section we fix $p,q>N,$ $p\not =q,$ and study the following Dirichlet
problem%
\begin{equation}
\left\{
\begin{array}
[c]{lll}%
-\left(  \Delta_{p}+\Delta_{q}\right)  u=\lambda\left\vert u(x_{u})\right\vert
^{p-2}u(x_{u})\delta_{x_{u}} & \mathrm{in} & \Omega\\
u=0 & \mathrm{on} & \partial\Omega,
\end{array}
\right.  \label{pqinf}%
\end{equation}
where $x_{u}$ is a maximum point of $\left\vert u\right\vert $ (so that
$\left\vert u(x_{u})\right\vert =\left\Vert u\right\Vert _{\infty}$) and
$\delta_{x_{u}}$ is the delta Dirac distribution supported at $x_{u}.$

As we will see in the sequel \eqref{pqinf} is the limit problem of (\ref{pqu})
as $r\rightarrow\infty.$

\begin{definition}
We say that $u\in X_{p,q}$ is a weak solution of \eqref{pqinf} if $\left\vert
u(x_{u})\right\vert =\left\Vert u\right\Vert _{\infty}$ and
\begin{equation}%
{\displaystyle\int_{\Omega}}
\left(  \left\vert \nabla u\right\vert ^{p-2}+\left\vert \nabla u\right\vert
^{q-2}\right)  \nabla u\cdot\nabla v\mathrm{d}x=\lambda\left\vert
u(x_{u})\right\vert ^{p-2}u(x_{u})v(x_{u}),\quad\forall\,v\in X_{p,q}.
\label{weakinf}%
\end{equation}

\end{definition}

Let us recall the Morrey's inequality, valid if $m>N:$%
\[
C\left\Vert u\right\Vert _{0,\alpha_{m}}^{m}\leq\left\Vert \nabla u\right\Vert
_{m}^{p},\quad\forall\,u\in W_{0}^{1,m}(\Omega),
\]
where $\left\Vert u\right\Vert _{0,s}$ denotes the standard norm in the
H\"{o}lder space $C^{0,s}(\overline{\Omega}),$ $\alpha_{m}=1-\frac{N}{m}$ and
the positive constant $C$ depends only on $\Omega,m$ and $N.$

Morrey's inequality implies that the embedding $W_{0}^{1,m}(\Omega
)\hookrightarrow C(\overline{\Omega})$ is compact and this fact guarantees
that the infimum of the Rayleigh quotient $\left\Vert \nabla v\right\Vert
_{m}^{m}/\left\Vert v\right\Vert _{\infty}^{m}$ is attained in $W_{0}%
^{1,m}(\Omega)\setminus\left\{  0\right\}  .$ From now on, we make use of the
additional notation%
\[
\lambda_{\infty}(m):=\min\left\{  \frac{\left\Vert \nabla v\right\Vert
_{m}^{m}}{\left\Vert v\right\Vert _{\infty}^{m}}:v\in W_{0}^{1,m}%
(\Omega)\setminus\left\{  0\right\}  \right\}  ,\quad m>N.
\]
As it is shown in \cite{EP16},
\begin{equation}
\lim_{r\rightarrow\infty}\lambda_{r}(m)=\lambda_{\infty}(m). \label{Lm}%
\end{equation}

\begin{remark}
\label{nonex}A nontrivial weak solution for \eqref{pqinf} cannot exist if
$\lambda\leq\lambda_{\infty}(p).$ Indeed, taking $v=u$ in \eqref{weakinf} one
has%
\[
\left(  \lambda-\lambda_{\infty}(p)\right)  \left\Vert u\right\Vert _{\infty
}^{p}=\left\Vert \nabla u\right\Vert _{p}^{p}+\left\Vert \nabla u\right\Vert
_{q}^{q}-\lambda_{\infty}(p)\left\Vert u\right\Vert _{\infty}^{p}>\left\Vert
\nabla u\right\Vert _{p}^{p}-\lambda_{\infty}(p)\left\Vert u\right\Vert
_{\infty}^{p}\geq0.
\]

\end{remark}

So, we assume in the rest of this section that $\lambda>\lambda_{\infty}(p).$

We define the energy functional $J_{\lambda}:X_{p,q}\rightarrow\mathbb{R}$
associated with (\ref{pqinf}) by%
\[
J_{\lambda}(u):=\frac{1}{p}\left\Vert \nabla u\right\Vert _{p}^{p}+\frac{1}%
{q}\left\Vert \nabla u\right\Vert _{q}^{q}-\frac{\lambda}{p}\left\Vert
u\right\Vert _{\infty}^{p}%
\]
and the Nehari set associated with $J_{\lambda}$ by
\[
\mathcal{N}_{\lambda,\infty}:=\left\{  u\in X_{p,q}\setminus\left\{
0\right\}  :\left\Vert \nabla u\right\Vert _{p}^{p}+\left\Vert \nabla
u\right\Vert _{q}^{q}=\lambda\left\Vert u\right\Vert _{\infty}^{p}\right\}  .
\]

Note that
\[
u\in\mathcal{N}_{\lambda,\infty}\Longrightarrow J_{\lambda}(u)=\left(
\frac{1}{q}-\frac{1}{p}\right)  \left\Vert \nabla u\right\Vert _{q}^{q}.
\]

Moreover, the identity
\[
\left\Vert \nabla(tu)\right\Vert _{p}^{p}+\left\Vert \nabla(tu)\right\Vert
_{q}^{q}-\lambda\left\Vert tu\right\Vert _{\infty}^{p}=t^{q}\left[  \left\Vert
\nabla u\right\Vert _{q}^{q}-t^{p-q}\left(  \lambda\left\Vert u\right\Vert
_{\infty}^{p}-\left\Vert \nabla u\right\Vert _{p}^{p}\right)  \right]  ,\quad
v\in X_{p,q},\quad\,t>0,
\]
allows us to derive the following equivalence, valid for the case
$N<q<p<\infty:$
\begin{equation}
tu\in\mathcal{N}_{\lambda,\infty}\Longleftrightarrow\lambda\left\Vert
u\right\Vert _{\infty}^{p}>\left\Vert \nabla u\right\Vert _{p}^{p}%
\quad\mathrm{and}\quad t=\left(  \frac{\left\Vert \nabla u\right\Vert _{q}%
^{q}}{\lambda\left\Vert u\right\Vert _{\infty}^{p}-\left\Vert \nabla
u\right\Vert _{p}^{p}}\right)  ^{\frac{1}{p-q}}. \label{vNehin}%
\end{equation}

Hence, by taking a function $e\in X_{p,q}\setminus\left\{  0\right\}  $ such
that $\left\Vert \nabla e\right\Vert _{p}^{p}=\lambda_{\infty}(p)\left\Vert
e\right\Vert _{\infty}^{p}$ we can see that $\mathcal{N}_{\lambda,\infty
}\not =\varnothing$ when $N<q<p<\infty.$

\begin{remark}
\label{infpos}In the case $N<q<p<\infty$ we also have%
\[
\mu_{\lambda}:=\inf_{u\in\mathcal{N}_{\lambda,\infty}}J_{\lambda}%
(u)\geq\left(  \frac{1}{q}-\frac{1}{p}\right)  \left(  \lambda^{-1}%
(\lambda_{\infty}(q))^{p/q}\right)  ^{\frac{q}{p-q}}>0.
\]
Indeed, this follows from the estimates
\[
\left\Vert \nabla u\right\Vert _{q}^{q}\leq\left\Vert \nabla u\right\Vert
_{p}^{p}+\left\Vert \nabla u\right\Vert _{q}^{q}=\lambda\left\Vert
u\right\Vert _{\infty}^{p}\leq\lambda(\lambda_{\infty}(q))^{-p/q}\left\Vert
\nabla u\right\Vert _{q}^{p},
\]
valid for any $u\in\mathcal{N}_{\lambda,\infty}.$
\end{remark}

\begin{definition}
We say that $u\in X_{p,q}$ is a least energy solution of \eqref{pqinf} if $u$
minimizes the functional $J_{\lambda}$ either on $X_{p,q}$ in the case
$N<p<q<\infty$ or on $\mathcal{N}_{\lambda,\infty}$ in the case $N<q<p<\infty
.$
\end{definition}

The functional $J_{\lambda}$ is not differentiable because of the term
involving the $L^{\infty}$ norm. Even though we are able to show that least
energy solutions are weak solutions. Indeed, this fact is a consequence of the
following identity (see \cite[Chapter 11]{Ban} and \cite{HL16}) valid for all
$u\in C(\overline{\Omega})$ and that provides the right Gateaux derivative for
the functional $u\mapsto\left\Vert u\right\Vert _{\infty}^{p}:$
\begin{equation}
\lim_{\epsilon\rightarrow0^{+}}\frac{\left\Vert u+\epsilon v\right\Vert
_{\infty}^{p}-\left\Vert u\right\Vert _{\infty}^{p}}{\epsilon}=p\max\left\{
\left\vert u(x)\right\vert ^{p-2}u(x)v(x):x\in\Gamma_{u}\right\}
,\quad\forall\,v\in C(\overline{\Omega}), \label{Ban}%
\end{equation}
where%
\[
\Gamma_{u}:=\left\{  x\in\Omega:\left\vert u(x)\right\vert =\left\Vert
u\right\Vert _{\infty}\right\}  .
\]

\begin{proposition}
\label{wsol}Least energy solutions of \eqref{pqinf} are weak solutions of this problem.
\end{proposition}

\begin{proof}
First we consider the case $N<p<q<\infty.$ We have, for each $v\in
X_{p,q}\hookrightarrow C(\overline{\Omega}),$%
\begin{equation}
\lim_{\epsilon\rightarrow0^{+}}\frac{J_{\lambda}(u+\epsilon v)-J_{\lambda}%
(u)}{\epsilon}=\lim_{\epsilon\rightarrow0^{+}}\left(  A(\epsilon
)-B(\epsilon)\right)  \label{limJ}%
\end{equation}
where
\[
A(\epsilon):=\frac{1}{p}\frac{\left\Vert \nabla(u+\epsilon v)\right\Vert
_{p}^{p}-\left\Vert \nabla u\right\Vert _{p}^{p}}{\epsilon}+\frac{1}{q}%
\frac{\left\Vert \nabla(u+\epsilon v)\right\Vert _{q}^{q}-\left\Vert \nabla
u\right\Vert _{q}^{q}}{\epsilon}%
\]
and%
\[
B(\epsilon):=\frac{\lambda}{p}\frac{\left\Vert u+\epsilon v\right\Vert
_{\infty}^{p}-\left\Vert u\right\Vert _{\infty}^{p}}{\epsilon}.
\]

Taking into account that the first limit in (\ref{limJ}) is nonnegative
(because $u$ minimizes $J_{\lambda}$) and still considering that
\[
\lim_{\epsilon\rightarrow0^{+}}A(\epsilon)=%
{\displaystyle\int_{\Omega}}
\left(  \left\vert \nabla u\right\vert ^{p-2}+\left\vert \nabla u\right\vert
^{q-2}\right)  \nabla u\cdot\nabla v\mathrm{d}x
\]
and that, according to (\ref{Ban}),
\[
\lim_{\epsilon\rightarrow0^{+}}B(\epsilon)=\lambda\max\left\{  \left\vert
u(x)\right\vert ^{p-2}u(x)v(x):x\in\Gamma_{u}\right\}  ,
\]
we conclude that%
\[
\lambda\max\left\{  \left\vert u(x)\right\vert ^{p-2}u(x)v(x):x\in\Gamma
_{u}\right\}  \leq%
{\displaystyle\int_{\Omega}}
\left(  \left\vert \nabla u\right\vert ^{p-2}\nabla u+\left\vert \nabla
u\right\vert ^{q-2}\nabla u\right)  \cdot\nabla v\mathrm{d}x.
\]

The arbitrariness of $v\in X_{p,q}$ allows us to replace $v$ with $-v$ in the
above inequality and also get%
\[
\lambda\min\left\{  \left\vert u(x)\right\vert ^{p-2}u(x)v(x):x\in\Gamma
_{u}\right\}  \geq%
{\displaystyle\int_{\Omega}}
\left(  \left\vert \nabla u\right\vert ^{p-2}\nabla u+\left\vert \nabla
u\right\vert ^{q-2}\nabla u\right)  \cdot\nabla v\mathrm{d}x.
\]
These last two inequalities lead us to the following identity
\[
\left\vert u(x)\right\vert ^{p-2}u(x)v(x)=%
{\displaystyle\int_{\Omega}}
\left(  \left\vert \nabla u\right\vert ^{p-2}+\left\vert \nabla u\right\vert
^{q-2}\right)  \nabla u\cdot\nabla v\mathrm{d}x,\quad\forall\,v\in
X_{p,q}\quad\mathrm{and}\quad\forall\,x\in\Gamma_{u},
\]
which implies that $\Gamma_{u}$ is a singleton, say
\[
\Gamma_{u}=\left\{  x_{u}\right\}
\]
for some $x_{u}\in\Omega.$ Consequently,
\[%
{\displaystyle\int_{\Omega}}
\left(  \left\vert \nabla u\right\vert ^{p-2}+\left\vert \nabla u\right\vert
^{q-2}\right)  \nabla u\cdot\nabla v\mathrm{d}x=\left\vert u(x_{u})\right\vert
^{p-2}u(x_{u})v(x_{u}),\quad\forall\,v\in X_{p,q},
\]
which is (\ref{weakinf}) for $u.$

We now analyze the case $N<q<p<\infty.$ Let us take an arbitrary function
$v\in X_{p,q}.$

Since $u\in\mathcal{N}_{\lambda,\infty}$ we have $\left\Vert \nabla
u\right\Vert _{p}^{p}<\lambda\left\Vert u\right\Vert _{\infty}^{p}.$ Hence, we
can take $\delta>0$ such that $u+sv\not \equiv 0$ and%
\[
\left\Vert \nabla(u+sv)\right\Vert _{p}^{p}<\lambda\left\Vert u+sv\right\Vert
_{\infty}^{p},\quad\forall s\in(-\delta,\delta).
\]

Let $\tau:(-\delta,\delta)\rightarrow(0,\infty)$ be the function given by%
\[
\tau(s)=\left(  \frac{\left\Vert \nabla(u+sv)\right\Vert _{q}^{q}}%
{\lambda\left\Vert u+sv\right\Vert _{\infty}^{p}-\left\Vert \nabla
(u+sv)\right\Vert _{p}^{p}}\right)  ^{\frac{1}{p-q}},
\]
which is right differentiable at $s=0.$

We can see from (\ref{vNehin}) that $\tau(s)(u+sv)\in\mathcal{N}%
_{\lambda,\infty}$ for all $s\in(-\delta,\delta)$ and that $\tau(0)=1$.

Now, let us consider the function $\gamma:(-\delta,\delta)\rightarrow
\mathbb{R}$ defined by
\[
\gamma(s)=J_{\lambda}\left(  \tau(s)(u+sv)\right)  =\frac{\tau(s)^{p}}%
{p}\left\Vert \nabla(u+sv)\right\Vert _{p}^{p}+\frac{\tau(s)^{q}}{q}\left\Vert
\nabla\left(  (u+sv)\right)  \right\Vert _{q}^{q}-\frac{\lambda\tau(s)^{p}}%
{p}\left\Vert u+sv\right\Vert _{\infty}^{p}.
\]
According to (\ref{Ban}) this function is right differentiable at $s=0$ and
\begin{align*}
\gamma^{\prime}(0_{+})  &  =\tau^{\prime}(0_{+})\left(  \left\Vert \nabla
u\right\Vert _{p}^{p}+\left\Vert \nabla u\right\Vert _{q}^{q}\right)  +%
{\displaystyle\int_{\Omega}}
\left(  \left\vert \nabla u\right\vert ^{p-2}+\left\vert \nabla u\right\vert
^{q-2}\right)  \nabla u\cdot\nabla v\mathrm{d}x\\
&  -\lambda\max\left\{  \left\vert u(x)\right\vert ^{p-2}u(x)v(x):x\in
\Gamma_{u}\right\}  -\tau^{\prime}(0_{+})\lambda\left\Vert u\right\Vert
_{\infty}^{p}\\
&  =%
{\displaystyle\int_{\Omega}}
\left(  \left\vert \nabla u\right\vert ^{p-2}+\left\vert \nabla u\right\vert
^{q-2}\right)  \nabla u\cdot\nabla v\mathrm{d}x-\lambda\max\left\{  \left\vert
u(x)\right\vert ^{p-2}u(x)v(x):x\in\Gamma_{u}\right\}  ,
\end{align*}
where we have used that $\tau(0)=1$ and $\left\Vert \nabla u\right\Vert
_{p}^{p}+\left\Vert \nabla u\right\Vert _{q}^{q}-\lambda\left\Vert
u\right\Vert _{\infty}^{p}=0.$

Since $\gamma$ attains its minimum value at $s=0$ we have
\[
\gamma^{\prime}(0_{+})=\lim_{s\rightarrow0^{+}}\frac{\gamma(s)-\gamma(0)}%
{s}\geq0.
\]
Hence,%
\[
\lambda\max\left\{  \left\vert u(x)\right\vert ^{p-2}u(x)v(x):x\in\Gamma
_{u}\right\}  \leq%
{\displaystyle\int_{\Omega}}
\left(  \left\vert \nabla u\right\vert ^{p-2}+\left\vert \nabla u\right\vert
^{q-2}\right)  \nabla u\cdot\nabla v\mathrm{d}x.
\]

Taking into account the arbitrariness of $v$ we replace $v$ with $-v$ to get%
\begin{align*}
\lambda\min\left\{  \left\vert u(x)\right\vert ^{p-2}u(x)v(x):x\in\Gamma
_{u}\right\}   &  \geq%
{\displaystyle\int_{\Omega}}
\left(  \left\vert \nabla u\right\vert ^{p-2}+\left\vert \nabla u\right\vert
^{q-2}\right)  \nabla u\cdot\nabla v\mathrm{d}x\\
&  \geq\lambda\max\left\{  \left\vert u(x)\right\vert ^{p-2}u(x)v(x):x\in
\Gamma_{u}\right\}  ,
\end{align*}
so that%
\[
\min\left\{  \left\vert u(x)\right\vert ^{p-2}u(x)v(x):x\in\Gamma_{u}\right\}
=\max\left\{  \left\vert u(x)\right\vert ^{p-2}u(x)v(x):x\in\Gamma
_{u}\right\}  ,
\]
implying both that $\Gamma_{u}=\left\{  x_{u}\right\}  ,$ for some $x_{u}%
\in\Omega,$ and that
\[%
{\displaystyle\int_{\Omega}}
\left(  \left\vert \nabla u\right\vert ^{p-2}+\left\vert \nabla u\right\vert
^{q-2}\right)  \nabla u\cdot\nabla v\mathrm{d}x=\lambda\left\vert
u(x_{u})\right\vert ^{p-2}u(x_{u})v(x_{u}).
\]

\end{proof}

Now we are ready to show that in both cases $N<p<q<\infty$ and $N<q<p<\infty$
a nonnegative least energy solution of (\ref{pqinf}) can be obtained from the
least energy solutions of (\ref{pqu}) by a limit process, by making as
$r\rightarrow\infty.$ For this we observe from (\ref{Lm}), with $m=p,$ that if
$\lambda>\lambda_{\infty}(p)$ and $r_{n}\rightarrow\infty,$ then there exists
$n_{0}\in\mathbb{N}$ such that $\lambda_{r_{n}}(p)<\lambda$ for all $n\geq
n_{0}.$ Therefore, for each $n\geq n_{0}$ the boundary value problem
\begin{equation}
\left\{
\begin{array}
[c]{lll}%
-\left(  \Delta_{p}+\Delta_{q}\right)  u=\lambda\left\Vert u\right\Vert
_{r_{n}}^{p-r_{n}}\left\vert u\right\vert ^{r_{n}-2}u & \mathrm{in} & \Omega\\
u=0 & \mathrm{on} & \partial\Omega
\end{array}
\right.  \label{rnpq}%
\end{equation}
has at least one nonnegative least energy solution $u_{n}.$ Having this in
mind, we can assume that $n_{0}=1$ in the next proposition.

\begin{proposition}
\label{existpqinf}Let $\lambda>\lambda_{\infty}(p)$ and $r_{n}\rightarrow
\infty.$ Denote by $u_{n}$ a nonnegative least energy solution of
\eqref{rnpq}. There exists a subsequence of $\left\{  u_{n}\right\}  $
converging strongly in $X_{p,q}$ to a nonnegative least energy solution $u$ of \eqref{pqinf}.
\end{proposition}

\begin{proof}
First we consider $N<p<q<\infty,$ so that $X_{p,q}=W_{0}^{1,q}(\Omega).$
Since
\begin{align*}
\left\Vert \nabla u_{n}\right\Vert _{q}^{q}  &  \leq\left\Vert \nabla
u_{n}\right\Vert _{q}^{q}+\left\Vert \nabla u_{n}\right\Vert _{p}^{p}\\
&  =\lambda\left\Vert u_{n}\right\Vert _{r_{n}}^{p}\leq\frac{\lambda}%
{\lambda_{r_{n}}(p)}\left\Vert \nabla u_{n}\right\Vert _{p}^{p}\leq
\frac{\lambda}{\lambda_{r_{n}}(p)}\left\Vert \nabla u_{n}\right\Vert _{q}%
^{p}\left\vert \Omega\right\vert ^{\frac{q-p}{q}},
\end{align*}
we have%
\[
\left\Vert \nabla u_{n}\right\Vert _{q}\leq\left\vert \Omega\right\vert
^{\frac{1}{q}}\left(  \frac{\lambda}{\lambda_{r_{n}}(p)}\right)  ^{\frac
{1}{q-p}},
\]
implying thus that $\left\{  u_{n}\right\}  $ is bounded in $X_{p,q}.$
Therefore, up to relabeling the sequence $\left\{  r_{n}\right\}  ,$ we can
assume that there exists a nonnegative function $u\in X_{p,q}$ such that
$u_{n}\rightharpoonup u$ in $X_{p,q}$ and $u_{n}\rightarrow u$ uniformly in
$\overline{\Omega}.$

In order to prove that $u$ minimizes $J_{\lambda}$ globally we fix an
arbitrary function $v\in X_{p,q}\hookrightarrow C(\overline{\Omega}).$ We know
that
\[
\frac{1}{p}\left\Vert \nabla u_{n}\right\Vert _{p}^{p}+\frac{1}{q}\left\Vert
\nabla u_{n}\right\Vert _{q}^{q}-\frac{\lambda}{p}\left\Vert u_{n}\right\Vert
_{r_{n}}^{p}\leq\frac{1}{p}\left\Vert \nabla v\right\Vert _{p}^{p}+\frac{1}%
{q}\left\Vert \nabla v\right\Vert _{q}^{q}-\frac{\lambda}{p}\left\Vert
v\right\Vert _{r_{n}}^{p},
\]
so that%
\begin{align*}
J_{\lambda}(u_{n})  &  =\frac{1}{p}\left\Vert \nabla u_{n}\right\Vert _{p}%
^{p}+\frac{1}{q}\left\Vert \nabla u_{n}\right\Vert _{q}^{q}-\frac{\lambda}%
{p}\left\Vert u_{n}\right\Vert _{\infty}^{p}\\
&  \leq\frac{1}{p}\left\Vert \nabla v\right\Vert _{p}^{p}+\frac{1}%
{q}\left\Vert \nabla v\right\Vert _{q}^{q}-\frac{\lambda}{p}\left\Vert
v\right\Vert _{r_{n}}^{p}+\frac{\lambda}{p}\left\Vert u_{n}\right\Vert
_{r_{n}}^{p}-\frac{\lambda}{p}\left\Vert u_{n}\right\Vert _{\infty}^{p}\\
&  \leq\frac{1}{p}\left\Vert \nabla v\right\Vert _{p}^{p}+\frac{1}%
{q}\left\Vert \nabla v\right\Vert _{q}^{q}-\frac{\lambda}{p}\left\Vert
v\right\Vert _{r_{n}}^{p}+\frac{\lambda\left\Vert u_{n}\right\Vert _{\infty
}^{p}}{p}\left(  \left\vert \Omega\right\vert ^{\frac{p}{r_{n}}}-1\right) \\
&  =J_{\lambda}(v)+\frac{\lambda}{p}\left\Vert v\right\Vert _{\infty}%
^{p}-\frac{\lambda}{p}\left\Vert v\right\Vert _{r_{n}}^{p}+\frac
{\lambda\left\Vert u_{n}\right\Vert _{\infty}^{p}}{p}\left(  \left\vert
\Omega\right\vert ^{\frac{p}{r_{n}}}-1\right)  .
\end{align*}

Since $v\in C(\overline{\Omega})$ we have $\left\Vert v\right\Vert _{r_{n}%
}^{p}\rightarrow\left\Vert v\right\Vert _{\infty}^{p}.$ This fact and the
convergences $u_{n}\rightharpoonup u$ and $u_{n}\rightarrow u$ in
$C(\overline{\Omega})$) imply that
\begin{align*}
J_{\lambda}(u)  &  =\liminf_{n\rightarrow\infty}J_{\lambda}(u_{n})\\
&  \leq J_{\lambda}(v)+\frac{\lambda}{p}\lim_{n\rightarrow\infty}\left(
\left\Vert v\right\Vert _{\infty}^{p}-\left\Vert v\right\Vert _{r_{n}}%
^{p}\right)  +\lim_{n\rightarrow\infty}\frac{\lambda\left\Vert u_{n}%
\right\Vert _{\infty}^{p}}{p}\left(  \left\vert \Omega\right\vert ^{\frac
{p}{r_{n}}}-1\right)  =J_{\lambda}(v).
\end{align*}
That is, $u$ minimizes $J_{\lambda}$ globally.

Now, let us consider the case $N<q<p<\infty,$ so that $X_{p,q}=W_{0}%
^{1,p}(\Omega)$ and
\begin{equation}
\left(  \frac{1}{q}-\frac{1}{p}\right)  \left\Vert \nabla u_{n}\right\Vert
_{q}^{q}=I_{\lambda,r_{n}}(u_{n})\leq I_{\lambda,r_{n}}(v)=\left(  \frac{1}%
{q}-\frac{1}{p}\right)  \left\Vert \nabla v\right\Vert _{q}^{q},\quad
\forall\,v\in\mathcal{N}_{\lambda,r_{n}}. \label{rnqp2}%
\end{equation}

In order to show that $\left\{  u_{n}\right\}  $ is bounded in $X_{p,q}$ we
pick $e_{n}\in X_{p,q}\setminus\left\{  0\right\}  $ satisfying (\ref{er})
with $r=r_{n},$ that is, such that%
\[
\left\Vert e_{n}\right\Vert _{r_{n}}=1\quad\mathrm{and}\quad\left\Vert \nabla
e_{n}\right\Vert _{p}^{p}=\lambda_{r_{n}}(p).
\]
Since $\lambda_{r_{n}}(p)<\lambda,$ we have $\left\Vert \nabla e_{n}%
\right\Vert _{p}^{p}<\lambda\left\Vert e_{n}\right\Vert _{r_{n}}^{p}$ and
$t_{n}e_{n}\in\mathcal{N}_{\lambda,r_{n}},$ where%
\[
t_{n}=\left(  \frac{\left\Vert \nabla e_{n}\right\Vert _{q}^{q}}%
{\lambda\left\Vert e_{n}\right\Vert _{r_{n}}^{p}-\left\Vert \nabla
e_{n}\right\Vert _{p}^{p}}\right)  ^{1/(p-q)}=\frac{\left\Vert \nabla
e_{n}\right\Vert _{q}^{q/(p-q)}}{(\lambda-\lambda_{r_{n}}(p))^{1/(p-q)}}.
\]

Hence, applying (\ref{rnqp2}), exploring the expression of $t_{n}$ and using
the H\"{o}lder inequality we obtain
\begin{align*}
\left\Vert \nabla u_{n}\right\Vert _{q}^{q}  &  \leq\left\Vert \nabla
(t_{n}e_{n})\right\Vert _{q}^{q}\\
&  =\frac{\left\Vert \nabla e_{n}\right\Vert _{q}^{q^{2}/(p-q)}\left\Vert
\nabla e_{n}\right\Vert _{q}^{q}}{(\lambda-\lambda_{r_{n}}(p))^{q/(p-q)}}\\
&  =\frac{(\left\Vert \nabla e_{n}\right\Vert _{q}^{q})^{p/(p-q)}}%
{(\lambda-\lambda_{r_{n}}(p))^{q/(p-q)}}\\
&  \leq\frac{(\left\vert \Omega\right\vert ^{(p-q)/p}\left\Vert \nabla
e_{n}\right\Vert _{p}^{q})^{p/(p-q)}}{(\lambda-\lambda_{r_{n}}(p))^{q/(p-q)}%
}=\left\vert \Omega\right\vert \left(  \frac{\lambda_{r_{n}}(p)}%
{\lambda-\lambda_{r_{n}}(p)}\right)  ^{q/(p-q)}.
\end{align*}

Recalling that $u_{n}\in\mathcal{N}_{\lambda,r_{n}}$ we have%
\begin{align*}
\left\Vert \nabla u_{n}\right\Vert _{p}^{p}  &  \leq\left\Vert \nabla
u_{n}\right\Vert _{p}^{p}+\left\Vert \nabla u_{n}\right\Vert _{q}^{q}\\
&  =\lambda\left\Vert u_{n}\right\Vert _{r_{n}}^{p}\\
&  \leq\lambda(\lambda_{r_{n}}(q))^{-p/q}\left\Vert \nabla u_{n}\right\Vert
_{q}^{p}\leq\lambda(\lambda_{r_{n}}(q))^{-p/q}\left\vert \Omega\right\vert
\left(  \frac{\lambda_{r_{n}}(p)}{\lambda-\lambda_{r_{n}}(p)}\right)
^{p/(p-q)},
\end{align*}
which gives us the boundedness of $\left\{  u_{n}\right\}  $ in $X_{p,q}$
since
\[
(\lambda_{r_{n}}(q))^{-p/q}\left(  \frac{\lambda_{r_{n}}(p)}{\lambda
-\lambda_{r_{n}}(p)}\right)  ^{p/(p-q)}\rightarrow(\lambda_{\infty}%
(q))^{-p/q}\left(  \frac{\lambda_{\infty}(p)}{\lambda-\lambda_{\infty}%
(p)}\right)  ^{p/(p-q)}.
\]

Thus, up to relabeling the sequence $\left\{  r_{n}\right\}  $ we can assume
that there exists a nonnegative function $u\in X_{p,q}$ such that
$u_{n}\rightharpoonup u$ in $X_{p,q}$ and $u_{n}\rightarrow u$ uniformly in
$\overline{\Omega}.$

We recall from (\ref{lowr}) that
\[
\left\Vert u_{n}\right\Vert _{r_{n}}\geq\left(  \frac{\lambda_{r_{n}}%
(q)}{\lambda}\right)  ^{1/(p-q)}.
\]
Hence, since $\left\Vert u_{n}\right\Vert _{r_{n}}\leq\left\Vert
u_{n}\right\Vert _{\infty}\left\vert \Omega\right\vert ^{1/r_{n}},$ we have
\[
\left\Vert u\right\Vert _{\infty}=\lim_{n\rightarrow\infty}\left\vert
\Omega\right\vert ^{1/r_{n}}\left\Vert u_{n}\right\Vert _{\infty}\geq
\lim_{n\rightarrow\infty}\left(  \frac{\lambda_{r_{n}}(q)}{\lambda}\right)
^{1/(p-q)}=\left(  \frac{\lambda_{\infty}(q)}{\lambda}\right)  ^{1/(p-q)}>0,
\]
that is, $u\not \equiv 0.$ Using this fact and
\[
\left\Vert \nabla u_{n}\right\Vert _{q}^{q}+\left\Vert \nabla u_{n}\right\Vert
_{p}^{p}=\lambda\left\Vert u_{n}\right\Vert _{r_{n}}^{p}\leq\lambda\left\Vert
u_{n}\right\Vert _{\infty}^{p}\left\vert \Omega\right\vert ^{p/r_{n}}%
\]
we obtain%
\begin{align*}
\left\Vert \nabla u\right\Vert _{p}^{p}  &  <\left\Vert \nabla u\right\Vert
_{q}^{q}+\left\Vert \nabla u\right\Vert _{p}^{p}\\
&  \leq\liminf_{n\rightarrow\infty}(\left\Vert \nabla u_{n}\right\Vert
_{q}^{q}+\left\Vert \nabla u_{n}\right\Vert _{p}^{p})\leq\lim_{n\rightarrow
\infty}(\lambda\left\Vert u_{n}\right\Vert _{\infty}^{p}\left\vert
\Omega\right\vert ^{p/r_{n}})=\lambda\left\Vert u\right\Vert _{\infty}^{p}.
\end{align*}

It follows that $tu\in\mathcal{N}_{\lambda,\infty}$ where $\ \ $%
\[
0<t=\left(  \frac{\left\Vert \nabla u\right\Vert _{q}^{q}}{\lambda\left\Vert
u\right\Vert _{\infty}^{p}-\left\Vert \nabla u\right\Vert _{p}^{p}}\right)
^{1/(p-q)}\leq1.
\]

Let us fix an arbitrary function $v\in\mathcal{N}_{\lambda,\infty}.$ We know
that
\[
\lim_{n\rightarrow\infty}\lambda\left\Vert v\right\Vert _{r_{n}}^{p}%
=\lambda\left\Vert v\right\Vert _{\infty}^{p}\quad\mathrm{and}\quad\left\Vert
\nabla v\right\Vert _{p}^{p}<\left\Vert \nabla v\right\Vert _{q}%
^{q}+\left\Vert \nabla v\right\Vert _{p}^{p}=\lambda\left\Vert v\right\Vert
_{\infty}^{p}.
\]
Consequently, there exists $n_{0}$ such that
\[
\left\Vert \nabla v\right\Vert _{p}^{p}<\lambda\left\Vert v\right\Vert
_{r_{n}}^{p},\quad\forall\,n\geq n_{0}.
\]
This implies that $t_{n}v\in\mathcal{N}_{\lambda,r_{n}}$ for all $n\geq
n_{0},$ where%
\[
t_{n}:=\left(  \frac{\left\Vert \nabla v\right\Vert _{q}^{q}}{\lambda
\left\Vert v\right\Vert _{r_{n}}^{p}-\left\Vert \nabla v\right\Vert _{p}^{p}%
}\right)  ^{1/(p-q)}\rightarrow\left(  \frac{\left\Vert \nabla v\right\Vert
_{q}^{q}}{\lambda\left\Vert v\right\Vert _{\infty}^{p}-\left\Vert \nabla
v\right\Vert _{p}^{p}}\right)  ^{1/(p-q)}=1.
\]

Thus,%
\[
\left(  \frac{1}{q}-\frac{1}{p}\right)  \left\Vert \nabla u_{n}\right\Vert
_{q}^{q}=I_{\lambda,r_{n}}(u_{n})\leq I_{\lambda,r_{n}}(v)=\left(  \frac{1}%
{q}-\frac{1}{p}\right)  \left\Vert \nabla(t_{n}v)\right\Vert _{q}^{q}%
,\quad\forall\,n\geq n_{0},
\]
so that%
\[
\left\Vert \nabla u\right\Vert _{q}^{q}\leq\liminf_{n\rightarrow\infty
}\left\Vert \nabla u_{n}\right\Vert _{q}^{q}\leq\lim_{n\rightarrow\infty
}(t_{n})^{q}\left\Vert \nabla v\right\Vert _{q}^{q}=\left\Vert \nabla
v\right\Vert _{q}^{q}.
\]
Therefore,%
\begin{equation}
J_{\lambda}(tu)=t^{q}\left(  \frac{1}{q}-\frac{1}{p}\right)  \left\Vert \nabla
u\right\Vert _{q}^{q}\leq t^{q}\left(  \frac{1}{q}-\frac{1}{p}\right)
\left\Vert \nabla v\right\Vert _{q}^{q}=t^{q}J_{\lambda}(v),\quad\forall
\,v\in\mathcal{N}_{\lambda,\infty}. \label{reva}%
\end{equation}
Let $\left\{  v_{n}\right\}  \subset\mathcal{N}_{\lambda,\infty}$ be such
that
\[
\lim_{n\rightarrow\infty}J_{\lambda}(v_{n})=\mu_{\lambda}=\inf_{u\in
\mathcal{N}_{\lambda,\infty}}J_{\lambda}(u).
\]
According to Remark \ref{infpos}, $\mu_{\lambda}>0.$ Thus, taking into account
(\ref{reva}) we obtain
\[
0<\mu_{\lambda}\leq J_{\lambda}(tu)\leq\lim_{n\rightarrow\infty}%
t^{q}J_{\lambda}(v_{n})=t^{q}\mu_{\lambda}\leq\mu_{\lambda}.
\]
These inequalities imply that: $t=1,$ $u\in\mathcal{N}_{\lambda,\infty}$ and
$J_{\lambda}(u)=\mu_{\lambda}.$ We have then shown that $u$ is a nonnegative
least energy solution of \eqref{pqinf}.

In order to conclude this proof we show that, in both cases above considered,
$u_{n}\rightarrow u$ strongly in $X_{p,q},$ up to a subsequence. In fact,
recalling that
\begin{equation}%
{\displaystyle\int_{\Omega}}
\left(  \left\vert \nabla u_{n}\right\vert ^{p-2}+\left\vert \nabla
u_{n}\right\vert ^{q-2}\right)  \nabla u_{n}\cdot\nabla v\mathrm{d}%
x=\lambda\left\Vert u_{n}\right\Vert _{r_{n}}^{p-r_{n}}%
{\displaystyle\int_{\Omega}}
\left\vert u_{n}\right\vert ^{r_{n}-1}v\mathrm{d}x,\quad\forall\,v\in X_{p,q},
\label{rnpq1}%
\end{equation}
$u_{n}\rightharpoonup u$ and $u_{n}\rightarrow u$ uniformly, we can see that
\[
\left\vert \lambda\left\Vert u_{n}\right\Vert _{r_{n}}^{p-r_{n}}%
{\displaystyle\int_{\Omega}}
\left\vert u_{n}\right\vert ^{r_{n}-1}(u_{n}-u)\mathrm{d}x\right\vert
\leq\lambda\left\Vert u_{n}\right\Vert _{\infty}^{p-1}\left\vert
\Omega\right\vert ^{\frac{p}{r_{n}}}\left\Vert u_{n}-u\right\Vert _{\infty
}\rightarrow0.
\]
That is, the right-hand side of (\ref{rnpq1}), with $v=u_{n}-u,$ goes to zero
as $n\rightarrow\infty.$

It follows that
\begin{equation}
A_{n}:=%
{\displaystyle\int_{\Omega}}
\left(  \left\vert \nabla u_{n}\right\vert ^{p-2}+\left\vert \nabla
u_{n}\right\vert ^{q-2}\right)  \nabla u_{n}\cdot\nabla(u_{n}-u)\mathrm{d}%
x\rightarrow0. \label{rnpq2}%
\end{equation}

The weak convergence $u_{n}\rightharpoonup u$ in $X_{p,q}$ also implies that%
\begin{equation}
B_{n}:=%
{\displaystyle\int_{\Omega}}
\left(  \left\vert \nabla u\right\vert ^{p-2}+\left\vert \nabla u\right\vert
^{q-2}\right)  \nabla u\cdot\nabla(u_{n}-u)\mathrm{d}x\rightarrow0.
\label{rnpq3}%
\end{equation}

Hence, taking into account (\ref{rnpq2})-(\ref{rnpq3}), noting that
\[%
{\displaystyle\int_{\Omega}}
\left(  \left\vert \nabla u_{n}\right\vert ^{p-2}\nabla u_{n}-\left\vert
\nabla u\right\vert ^{p-2}\nabla u+\left\vert \nabla u_{n}\right\vert
^{q-2}\nabla u_{n}-\left\vert \nabla u\right\vert ^{q-2}\nabla u\right)
\cdot\nabla(u_{n}-u)\mathrm{d}x=A_{n}-B_{n}%
\]
and recalling the following well-known inequality, valid for all $\xi,\eta
\in\mathbb{R}^{N}$ and $m\geq2,$%
\begin{equation}%
{\displaystyle\int_{\Omega}}
\left(  \left\vert \nabla\xi\right\vert ^{m-2}\nabla\xi-\left\vert \nabla
\eta\right\vert ^{m-2}\nabla\eta\right)  \cdot\nabla(\xi-\eta)\mathrm{d}%
x\geq2^{2-m}%
{\displaystyle\int_{\Omega}}
\left\vert \xi-\eta\right\vert ^{m}\mathrm{d}x \label{desvec}%
\end{equation}
we conclude that
\[
\left\Vert \nabla(u_{n}-u)\right\Vert _{q}\rightarrow0\quad\mathrm{and}%
\quad\left\Vert \nabla(u_{n}-u)\right\Vert _{p}\rightarrow0.
\]
Thus, $u_{n}\rightarrow u$ strongly in $X_{p,q}.$
\end{proof}

\section{The limit problem as $p\rightarrow\infty$\label{sec4}}

It is proved in \cite{EP16} that
\[
\lim_{m\rightarrow\infty}\left(  \lambda_{\infty}(m)\right)  ^{1/m}%
=\Lambda_{\infty},
\]
where $\Lambda_{\infty}$ is defined in (\ref{Laminf}). We recall that (see
\cite{JLM})
\[
\Lambda_{\infty}=\left\Vert \rho\right\Vert _{\infty}^{-1}%
\]
where $\rho:\overline{\Omega}\rightarrow\mathbb{R}$ denotes the distance
function to the boundary, given by%
\[
\rho(x)=\inf\left\{  \left\vert x-y\right\vert :y\in\partial\Omega\right\}  .
\]

We recall two well-known facts: $\left\vert \nabla\rho\right\vert =1$ almost
everywhere in $\Omega$ and $\rho\in W^{1,\infty}(\Omega)\cap C_{0}%
(\overline{\Omega})\subset W_{0}^{1,m}(\Omega)$ for all $m\in\lbrack
1,\infty).$

\begin{lemma}
Let $\lambda>\lambda_{\infty}(p)$ and consider $u$ a nonnegative least energy
solution of the boundary value problem%
\[
\left\{
\begin{array}
[c]{lll}%
-\left(  \Delta_{p}+\Delta_{q}\right)  u=\lambda\left\Vert u\right\Vert
_{\infty}^{p-1}\delta_{x_{u}} & \mathrm{in} & \Omega\\
u=0 & \mathrm{on} & \partial\Omega.
\end{array}
\right.
\]
Then%
\begin{equation}
\left\Vert \nabla u\right\Vert _{q}\leq\left\vert \Omega\right\vert ^{\frac
{1}{q}}\left(  \frac{\lambda_{\infty}(p)}{\lambda-\lambda_{\infty}(p)}\right)
^{\frac{1}{p-q}},\quad\mathrm{if}\,N<q<p, \label{estqp}%
\end{equation}
and
\begin{equation}
\left\Vert \nabla u\right\Vert _{q}\leq\left\vert \Omega\right\vert ^{\frac
{1}{q}}\left(  \frac{\lambda}{\lambda_{\infty}(p)}\right)  ^{\frac{1}{q-p}%
},\quad\mathrm{if}\,N<p<q. \label{estpq}%
\end{equation}

\end{lemma}

\begin{proof}
First we consider the case $N<q<p.$ Let $e\in X_{p,q}=W_{0}^{1,p}(\Omega)$ be
such that
\[
\left\Vert e\right\Vert _{\infty}=1\quad\mathrm{and}\quad\left\Vert \nabla
e\right\Vert _{p}^{p}=\lambda_{\infty}(p).
\]
Since
\[
\lambda\left\Vert e\right\Vert _{\infty}^{p}-\left\Vert \nabla e\right\Vert
_{p}^{p}=\lambda-\lambda_{\infty}(p)>0
\]
we have $te\in N_{\lambda,\infty},$ where%
\[
t:=\left(  \frac{\left\Vert \nabla e\right\Vert _{q}^{q}}{\lambda\left\Vert
e\right\Vert _{\infty}^{p}-\left\Vert \nabla e\right\Vert _{p}^{p}}\right)
^{1/(p-q)}=\left(  \frac{\left\Vert \nabla e\right\Vert _{q}^{q}}%
{\lambda-\lambda_{\infty}(p)}\right)  ^{1/(p-q)}.
\]

Noting that%
\[
0<\left(  \frac{1}{q}-\frac{1}{p}\right)  \left\Vert \nabla u\right\Vert
_{q}^{q}=I_{\lambda,\infty}(u)\leq I_{\lambda,\infty}(te)=\left(  \frac{1}%
{q}-\frac{1}{p}\right)  \left\Vert \nabla(te)\right\Vert _{q}^{q}%
\]
we obtain (by exploring the expression of $t$ and using the H\"{o}lder
inequality)
\begin{align*}
\left\Vert \nabla u\right\Vert _{q}^{q}  &  \leq\left\Vert \nabla
(te)\right\Vert _{q}^{q}\\
&  =\frac{\left\Vert \nabla e\right\Vert _{q}^{q^{2}/(p-q)}\left\Vert \nabla
e\right\Vert _{q}^{q}}{(\lambda-\lambda_{\infty}(p))^{q/(p-q)}}\\
&  =\frac{(\left\Vert \nabla e\right\Vert _{q}^{q})^{p/(p-q)}}{(\lambda
-\lambda_{\infty}(p))^{q/(p-q)}}\leq\frac{(\left\vert \Omega\right\vert
^{(p-q)/p}\left\Vert \nabla e\right\Vert _{p}^{q})^{p/(p-q)}}{(\lambda
-\lambda_{\infty}(p))^{q/(p-q)}}=\left\vert \Omega\right\vert \left(
\frac{\lambda_{\infty}(p)}{\lambda-\lambda_{\infty}(p)}\right)  ^{q/(p-q)}.
\end{align*}
This leads to the estimate in (\ref{estqp}).

The estimate in (\ref{estpq}) is a direct consequence of the following
\begin{align*}
\left\Vert \nabla u\right\Vert _{q}^{q}  &  \leq\left\Vert \nabla u\right\Vert
_{q}^{q}+\left\Vert \nabla u\right\Vert _{p}^{p}\\
&  =\lambda\left\Vert u\right\Vert _{\infty}^{p}\leq\frac{\lambda}%
{\lambda_{\infty}(p)}\left\Vert \nabla u\right\Vert _{p}^{p}\leq\frac{\lambda
}{\lambda_{\infty}(p)}\left\vert \Omega\right\vert ^{\frac{q-p}{q}}\left\Vert
\nabla u\right\Vert _{q}^{p}.
\end{align*}

\end{proof}

We recall that
\[
\lim_{p\rightarrow\infty}\frac{q(p)}{p}=\left\{
\begin{array}
[c]{lll}%
Q\in(0,1) & \mathrm{if} & N<q<p\\
Q\in(1,\infty) & \mathrm{if} & N<p<q.
\end{array}
\right.
\]

\begin{lemma}
\label{lemapqdelta}Let $\Lambda>\Lambda_{\infty}$ and $m>N$ be fixed. Take
$\lambda_{p}>0$ satisfying%
\[
\lim_{p\rightarrow\infty}(\lambda_{p})^{\frac{1}{p}}=\Lambda
\]
and denote by $u_{p}$ a nonnegative least energy solution of \
\begin{equation}
\left\{
\begin{array}
[c]{lll}%
-\left(  \Delta_{p}+\Delta_{q(p)}\right)  u=\lambda_{p}\left\Vert u\right\Vert
_{\infty}^{p-1}\delta_{x_{u}} & \mathrm{in} & \Omega\\
u=0 & \mathrm{on} & \partial\Omega.
\end{array}
\right.  \label{pqdelta}%
\end{equation}
We affirm that
\begin{equation}
\limsup_{p\rightarrow\infty}\left\Vert \nabla u_{p}\right\Vert _{m}%
\leq\left\vert \Omega\right\vert ^{\frac{1}{m}}\left(  \frac{\Lambda_{\infty}%
}{\Lambda}\right)  ^{\frac{1}{1-Q}}\label{LS}%
\end{equation}
and
\begin{equation}
\liminf_{p\rightarrow\infty}\left\Vert u_{p}\right\Vert _{\infty}\geq\left\{
\begin{array}
[c]{lll}%
(\Lambda_{\infty})^{-1}\left(  \Lambda_{\infty}/\Lambda\right)  ^{\frac
{1}{1-Q}} & \mathrm{if} & Q\in(0,1)\\
&  & \\
(\Lambda_{\infty})^{-1} & \mathrm{if} & Q\in(1,\infty).
\end{array}
\right.  \label{LI}%
\end{equation}

\end{lemma}

\begin{proof}
Since
\[
\lim_{p\rightarrow\infty}\left(  \lambda_{\infty}(p)\right)  ^{1/p}%
=\Lambda_{\infty}<\Lambda=\lim_{p\rightarrow\infty}\left(  \lambda_{p}\right)
^{1/p},
\]
we can see that $\lambda_{\infty}(p)<\lambda_{p}$ for all $p$ large enough.
Therefore, the existence of a least energy solution $u_{p}$ follows from
Proposition \ref{existpqinf}.

Let us fix $p_{n}\rightarrow\infty$ and simplify the notation by defining
\[
u_{n}:=u_{p_{n}},\quad q_{n}:=q(p_{n})\quad\mathrm{and}\quad\lambda
_{n}:=\lambda_{p_{n}}.
\]
Let $n_{0}\in\mathbb{N}$ such that $m<\min\left\{  q_{n},p_{n}\right\}  $ for
all $n\geq n_{0}.$ Now, fix $0<\epsilon<(\Lambda/\Lambda_{\infty})-1$ and
consider $n_{1}\geq n_{0}$ such that%
\[
1<a_{\epsilon}:=\frac{\Lambda}{\Lambda_{\infty}}-\epsilon\leq\left(
\frac{\lambda_{n}}{\lambda_{\infty}(p_{n})}\right)  ^{\frac{1}{p_{n}}}%
\leq\frac{\Lambda}{\Lambda_{\infty}}+\epsilon=:b_{\epsilon},\quad
\forall\,n\geq n_{1}.
\]

First we prove (\ref{LS}) in the case $Q\in(0,1),$ so that $N<q_{n}<p_{n}.$
Thus, according to (\ref{estqp}), with $\lambda=\lambda_{n},$ we have%
\begin{equation}
\left\Vert \nabla u_{n}\right\Vert _{q_{n}}\leq\left\vert \Omega\right\vert
^{1/q_{n}}\left(  \frac{\lambda_{\infty}(p_{n})}{\lambda_{n}-\lambda_{\infty
}(p_{n})}\right)  ^{1/(p_{n}-q_{n})}=\left\vert \Omega\right\vert ^{1/q_{n}%
}\left(  \frac{1}{(\lambda_{n}/\lambda_{\infty}(p_{n}))-1}\right)
^{1/(p_{n}-q_{n})}. \label{LS1}%
\end{equation}

Applying the H\"{o}lder inequality in (\ref{LS1})
\begin{align*}
\left\Vert \nabla u_{n}\right\Vert _{m}  &  \leq\left\vert \Omega\right\vert
^{1/m-1/q_{n}}\left\Vert \nabla u_{n}\right\Vert _{q_{n}}\\
&  \leq\left\vert \Omega\right\vert ^{1/m-1/q_{n}}\left\vert \Omega\right\vert
^{1/q_{n}}\left(  \frac{1}{(\lambda_{n}/\lambda_{\infty}(p_{n}))-1}\right)
^{1/(p_{n}-q_{n})}\\
&  \leq\left\vert \Omega\right\vert ^{1/m}\left(  \frac{1}{(a_{\epsilon
})^{p_{n}}-1}\right)  ^{1/(p_{n}-q_{n})},\quad\forall\,n\geq n_{1},
\end{align*}
Hence,
\begin{align*}
\limsup_{n\rightarrow\infty}\left\Vert \nabla u_{n}\right\Vert _{m}  &
\leq\left\vert \Omega\right\vert ^{1/m}\lim_{p\rightarrow\infty}\left(
(a_{\epsilon})^{p_{n}}-1\right)  ^{-1/(p_{n}-q_{n})}\\
&  =\left\vert \Omega\right\vert ^{1/m}\lim_{p\rightarrow\infty}\left(
(a_{\epsilon})^{p_{n}}-1\right)  ^{-\frac{1}{p_{n}}\frac{1}{1-(q_{n}/p_{n})}%
}=\left\vert \Omega\right\vert ^{1/m}(a_{\epsilon})^{1/(1-Q)}%
\end{align*}
since%
\[
\lim_{p\rightarrow\infty}\left(  (a_{\epsilon})^{p}-1\right)  ^{-1/p}%
=\lim_{p\rightarrow\infty}\exp\left(  -\frac{1}{p}\log\left(  (a_{\epsilon
})^{p}-1\right)  \right)  =a_{\epsilon}.
\]
Letting $\epsilon\rightarrow0,$ we obtain (\ref{LS}) when $Q\in(0,1).$

Now, we prove (\ref{LS}) when $Q\in(1,\infty),$ in which case $N<p_{n}<q_{n}.$
By the H\"{o}lder inequality and (\ref{estpq}), with $\lambda=\lambda_{n},$ we
have%
\begin{align*}
\left\Vert \nabla u_{n}\right\Vert _{m}  &  \leq\left\vert \Omega\right\vert
^{1/m-1/q_{n}}\left\Vert \nabla u_{n}\right\Vert _{q_{n}}\\
&  \leq\left\vert \Omega\right\vert ^{1/m-1/q_{n}}\left\vert \Omega\right\vert
^{1/q_{n}}\left(  \frac{\lambda_{n}}{\lambda_{\infty}(p_{n})}\right)
^{\frac{1}{q_{n}-p_{n}}}\\
&  \leq\left\vert \Omega\right\vert ^{1/m}\left(  b_{\epsilon}\right)
^{p_{n}\frac{1}{q_{n}-p_{n}}},\quad\forall\,n\geq n_{1}.
\end{align*}

Therefore,
\[
\limsup_{n\rightarrow\infty}\left\Vert \nabla u_{n}\right\Vert _{m}%
\leq\left\vert \Omega\right\vert ^{1/m}\lim_{n\rightarrow\infty}\left(
b_{\epsilon}\right)  ^{p_{n}\frac{1}{q_{n}-p_{n}}}=\left\vert \Omega
\right\vert ^{1/m}\lim_{n\rightarrow\infty}\left(  b_{\epsilon}\right)
^{\frac{1}{(q_{n}/p_{n})-1}}=\left\vert \Omega\right\vert ^{1/m}\left(
b_{\epsilon}\right)  ^{1/Q-1}.
\]
Letting $\epsilon\rightarrow0,$ we also obtain (\ref{LS}) when $Q\in
(1,\infty).$

Let us pass to the proof of (\ref{LI}). In the case $Q\in(0,1)$, in which
$N<q_{n}<p_{n},$ we have%
\begin{align*}
\left\Vert u_{n}\right\Vert _{\infty}^{q_{n}}  &  \leq\frac{1}{\lambda
_{\infty}(q_{n})}\left\Vert \nabla u_{n}\right\Vert _{q_{n}}^{q_{n}}\\
&  \leq\frac{1}{\lambda_{\infty}(q_{n})}\left(  \left\Vert \nabla
u_{n}\right\Vert _{q_{n}}^{q_{n}}+\left\Vert \nabla u_{n}\right\Vert _{p_{n}%
}^{p_{n}}\right) \\
&  =\frac{\lambda_{n}}{\lambda_{\infty}(q_{n})}\left\Vert u_{n}\right\Vert
_{\infty}^{p_{n}}\leq\left(  b_{\epsilon}\right)  ^{p_{n}}\frac{\lambda
_{\infty}(p_{n})}{\lambda_{\infty}(q_{n})}\left\Vert u_{n}\right\Vert
_{\infty}^{p_{n}}.
\end{align*}
It follows that
\begin{align*}
\liminf_{n\rightarrow\infty}\left\Vert u_{n}\right\Vert _{\infty}  &  \geq
\lim_{n\rightarrow\infty}\left(  \left(  b_{\epsilon}\right)  ^{-p_{n}}%
\frac{\lambda_{\infty}(q_{n})}{\lambda_{\infty}(p_{n})}\right)  ^{1/(p_{n}%
-q_{n})}\\
&  =\lim_{n\rightarrow\infty}\left(  b_{\epsilon}\right)  ^{-\frac{1}%
{1-(q_{n}/p_{n})}}\lim_{n\rightarrow\infty}\left(  \lambda_{\infty}%
(q_{n})^{\frac{1}{q_{n}}}\right)  ^{q_{n}/(p_{n}-q_{n})}\lim_{n\rightarrow
\infty}\left(  \lambda_{\infty}(p_{n})^{-\frac{1}{p_{n}}}\right)
^{p_{n}/(p_{n}-q_{n})}\\
&  =\left(  b_{\epsilon}\right)  ^{-\frac{1}{1-Q}}(\Lambda_{\infty}%
)^{Q/(1-Q)}(\Lambda_{\infty})^{-1/(1-Q)}=\left(  b_{\epsilon}\right)
^{-\frac{1}{1-Q}}(\Lambda_{\infty})^{-1}.
\end{align*}
Thus, making $\epsilon\rightarrow0$ we obtain (\ref{LI}) in the case
$Q\in(0,1).$

As for the case $Q\in(1,\infty)$, in which $N<p_{n}<q_{n},$ we have
\[
\left(  \frac{1}{q_{n}}-\frac{1}{p_{n}}\right)  \left\Vert \nabla
u_{n}\right\Vert _{q_{n}}^{q_{n}}=I_{\lambda_{n},\infty}(u_{n})\leq
I_{\lambda_{n},\infty}(\rho)=\frac{\left\vert \Omega\right\vert }{q_{n}}%
+\frac{\left\vert \Omega\right\vert }{p_{n}}-\frac{\lambda_{n}}{p_{n}%
}\left\Vert \rho\right\Vert _{\infty}^{p_{n}}%
\]
since $\rho\in X_{p_{n},q_{n}}=W_{0}^{1,q_{n}}(\Omega)$ and $\left\vert
\nabla\rho\right\vert =1$ almost everywhere.

Hence, since $\left\Vert \nabla u_{n}\right\Vert _{q_{n}}^{q_{n}}%
\leq\left\Vert \nabla u_{n}\right\Vert _{p_{n}}^{p_{n}}+\left\Vert \nabla
u_{n}\right\Vert _{q_{n}}^{q_{n}}=\lambda_{n}\left\Vert u_{n}\right\Vert
_{\infty}^{p_{n}}$ and $\left\Vert \rho\right\Vert _{\infty}^{-1}%
=\Lambda_{\infty},$ we obtain%
\[
\frac{\lambda_{n}}{p_{n}(\Lambda_{\infty})^{p_{n}}}\leq\left\vert
\Omega\right\vert \left(  \frac{1}{p_{n}}+\frac{1}{q_{n}}\right)  +\left(
\frac{1}{p_{n}}-\frac{1}{q_{n}}\right)  \lambda_{n}\left\Vert u_{n}\right\Vert
_{\infty}^{p_{n}},
\]
so that
\[
\left\Vert u_{n}\right\Vert _{\infty}^{p_{n}}\geq\left(  1-\frac{p_{n}}{q_{n}%
}\right)  ^{-1}\left[  \frac{1}{(\Lambda_{\infty})^{p_{n}}}-\frac{\left\vert
\Omega\right\vert }{\lambda_{n}}\left(  1+\frac{p_{n}}{q_{n}}\right)  \right]
.
\]

Since
\[
\left[  \frac{\left\vert \Omega\right\vert}{\lambda_{n}} \left(
1+\frac{p_{n}}{q_{n}}\right)  \right]  ^{\frac{1}{p_{n}}}\rightarrow\frac
{1}{\Lambda}<\frac{1}{\Lambda_{\infty}},\quad\forall\,n\geq n_{1}
\]
we can assume that%
\[
\frac{\left\vert \Omega\right\vert }{\lambda_{n}}\left(  1+\frac{p_{n}}{q_{n}%
}\right)  \leq\left[  \frac{1}{2}\left(  \frac{1}{\Lambda}+\frac{1}%
{\Lambda_{\infty}}\right)  \right]  ^{p_{n}}=\left(  \frac{\Lambda_{\infty
}+\Lambda}{2\Lambda\Lambda_{\infty}}\right)  ^{p_{n}},\quad\forall\,n\geq
n_{1}.
\]

Hence,
\begin{align*}
\left\Vert u_{n}\right\Vert _{\infty}  & \geq\left(  1-\frac{p_{n}}{q_{n}%
}\right)  ^{-\frac{1}{p_{n}}}\left[  \frac{1}{(\Lambda_{\infty})^{p_{n}}%
}-\left(  \frac{\Lambda_{\infty}+\Lambda}{2\Lambda\Lambda_{\infty}}\right)
^{p_{n}}\right]  ^{\frac{1}{p_{n}}}\\
& =\left(  1-\frac{p_{n}}{q_{n}}\right)  ^{-\frac{1}{p_{n}}}\frac{1}%
{\Lambda_{\infty}}\left[  1-\left(  \frac{\Lambda_{\infty}+\Lambda}{2\Lambda
}\right)  ^{p_{n}}\right]  ^{\frac{1}{p_{n}}}>0,\quad\forall\,n\geq n_{1}.
\end{align*}
Therefore,
\[
\liminf_{p\rightarrow\infty}\left\Vert u_{n}\right\Vert _{\infty}\geq\frac
{1}{\Lambda_{\infty}}\lim_{p\rightarrow\infty}\left(  1-\frac{p_{n}}{q_{n}%
}\right)  ^{-\frac{1}{p_{n}}}\left[  1-\left(  \frac{\Lambda_{\infty}+\Lambda
}{2\Lambda}\right)  ^{p_{n}}\right]  ^{\frac{1}{p_{n}}}=\frac{1}%
{\Lambda_{\infty}}.
\]

\end{proof}

\begin{theorem}
\label{uinf}Let $\Lambda>\Lambda_{\infty}$ be fixed and take $\lambda_{p}>0$
satisfying
\[
\lim_{p\rightarrow\infty}(\lambda_{p})^{\frac{1}{p}}=\Lambda.
\]
Denote by $u_{p}$ a nonnegative least energy solution of \eqref{pqdelta} and
by $x_{p}$ the only maximum point of $u_{p}$ (that is $x_{p}:=x_{u_{p}}$).
There exists a sequence $p_{n}\rightarrow\infty,$ a point $x_{\Lambda}%
\in\Omega$ and a function $u_{\Lambda}\in W^{1,\infty}(\Omega)\cap
C_{0}(\overline{\Omega})$ such that $x_{p_{n}}\rightarrow x_{\Lambda}$ and
$u_{p_{n}}\rightarrow u_{\Lambda}$ uniformly in $\overline{\Omega}.$ Moreover,%
\begin{equation}
\left\Vert \nabla u_{\Lambda}\right\Vert _{\infty}\leq\left(  \frac
{\Lambda_{\infty}}{\Lambda}\right)  ^{\frac{1}{1-Q}} \label{LSi}%
\end{equation}
and%
\begin{equation}
u_{\Lambda}(x_{\Lambda})=\left\Vert u_{\Lambda}\right\Vert _{\infty}%
\geq\left\{
\begin{array}
[c]{lll}%
(\Lambda_{\infty})^{-1}\left(  \Lambda_{\infty}/\Lambda\right)  ^{\frac
{1}{1-Q}} & \mathrm{if} & Q\in(0,1)\\
&  & \\
(\Lambda_{\infty})^{-1} & \mathrm{if} & Q\in(1,\infty).
\end{array}
\right.  \label{LIi}%
\end{equation}

\end{theorem}

\begin{proof}
Let $p_{n}\rightarrow\infty$ and $N<m<\infty.$ It follows from the previous
lemma that $\left\{  u_{p_{n}}\right\}  $ is bounded in $W_{0}^{1,m}(\Omega)$.
Thus, up to a subsequence, $u_{p_{n}}$ converges weakly in $W_{0}^{1,m}%
(\Omega)$ and uniformly in $\overline{\Omega}$ to a nonnegative function
$u_{\Lambda}\in W_{0}^{1,m}(\Omega)\cap C(\overline{\Omega}).$ Therefore, in
view of (\ref{LS}) we have
\[
\left\Vert \nabla u_{\Lambda}\right\Vert _{m}\leq\liminf_{n\rightarrow\infty
}\left\Vert \nabla u_{p_{n}}\right\Vert _{m}\leq\limsup_{n\rightarrow\infty
}\left\Vert \nabla u_{p_{n}}\right\Vert _{m}\leq\left\vert \Omega\right\vert
^{\frac{1}{m}}\left(  \frac{\Lambda_{\infty}}{\Lambda}\right)  ^{\frac{1}%
{1-Q}}.
\]
Hence, noting that $m\in(N,\infty)$ is arbitrary, we conclude that
$u_{\Lambda}\in W^{1,\infty}(\Omega)$ and%
\[
\left\Vert \nabla u_{\Lambda}\right\Vert _{\infty}\leq\lim_{m\rightarrow
\infty}\left\vert \Omega\right\vert ^{\frac{1}{m}}\left(  \frac{\Lambda
_{\infty}}{\Lambda}\right)  ^{\frac{1}{Q-1}}=\left(  \frac{\Lambda_{\infty}%
}{\Lambda}\right)  ^{\frac{1}{1-Q}},
\]
which is (\ref{LSi}).

The uniform convergence and (\ref{LI}) imply (\ref{LIi}), which in turn, shows
that $\left\Vert u_{\Lambda}\right\Vert _{\infty}>0.$ Taking into account that
$\left\{  x_{p_{n}}\right\}  $ is bounded, we can assume (up to relabeling the
sequence $\left\{  p_{n}\right\}  $) that $x_{p_{n}}\rightarrow x_{\Lambda}$
for some $x_{\Lambda}\in\overline{\Omega}.$ The uniform convergence also
implies that $u_{\Lambda}(x_{\Lambda})=\left\Vert u_{\Lambda}\right\Vert
_{\infty}>0$ so that $x_{\Lambda}\in\Omega$ (note that $u_{\Lambda}\equiv0$ on
$\partial\Omega$).
\end{proof}

The next corollary shows that in the case $Q\in(0,1)$ the function
$u_{\Lambda}$, such as $\rho,$ minimizes the Rayleigh quotient $\left\Vert
\nabla v\right\Vert _{\infty}/\left\Vert v\right\Vert _{\infty}$ in $\left(
W^{1,\infty}(\Omega)\cap C_{0}(\overline{\Omega})\right)  \setminus\left\{
0\right\}  .$

\begin{corollary}
\label{corol}If $Q\in(0,1),$ then
\begin{equation}
\left\Vert u_{\Lambda}\right\Vert _{\infty}=\frac{1}{\Lambda_{\infty}}\left(
\frac{\Lambda_{\infty}}{\Lambda}\right)  ^{\frac{1}{1-Q}}\quad\mathrm{and}%
\quad\Lambda_{\infty}=\frac{\left\Vert \nabla u_{\Lambda}\right\Vert _{\infty
}}{\left\Vert u_{\Lambda}\right\Vert _{\infty}},\quad\forall\,\Lambda
>\Lambda_{\infty}. \label{uextrema}%
\end{equation}
Therefore, $x_{\Lambda}$ is also a maximum point of the distance function to
the boundary $\rho$ and%
\begin{equation}
0\leq u_{\Lambda}(x)\leq\left(  \frac{\Lambda_{\infty}}{\Lambda}\right)
^{\frac{1}{1-Q}}\rho(x)\quad\forall\,x\in\overline{\Omega}, \label{uinfrho}%
\end{equation}
with the equality holding in $\partial\Omega\cup\left\{  x_{\Lambda}\right\}
.$
\end{corollary}

\begin{proof}
According to (\ref{LIi}) and (\ref{LSi}) we have$,$%
\[
\frac{1}{\Lambda_{\infty}}\left(  \frac{\Lambda_{\infty}}{\Lambda}\right)
^{\frac{1}{1-Q}}\leq\left\Vert u_{\Lambda}\right\Vert _{\infty}\leq
\frac{\left\Vert \nabla u_{\Lambda}\right\Vert _{\infty}}{\Lambda_{\infty}%
}\leq\frac{1}{\Lambda_{\infty}}\left(  \frac{\Lambda_{\infty}}{\Lambda
}\right)  ^{\frac{1}{1-Q}},
\]
which gives (\ref{uextrema}).

Taking into account that $\left\Vert \nabla u_{\Lambda}\right\Vert _{\infty
}=\Lambda_{\infty}\left\Vert u_{\Lambda}\right\Vert _{\infty}=\left\Vert
\rho\right\Vert _{\infty}^{-1}\left\Vert u_{\Lambda}\right\Vert _{\infty}$ we
have
\[
0\leq u_{\Lambda}(x)=u_{\Lambda}(x)-u_{\Lambda}(y)\leq\left\Vert \nabla
u_{\Lambda}\right\Vert _{\infty}\left\vert x-y\right\vert =\left\Vert
\rho\right\Vert _{\infty}^{-1}\left\Vert u_{\Lambda}\right\Vert _{\infty
}\left\vert x-y\right\vert \,
\]
for each $x\in\overline{\Omega}$ and $y\in\partial\Omega.$ It follows that
\[
0\leq\frac{\left\Vert \rho\right\Vert _{\infty}}{\left\Vert u_{\Lambda
}\right\Vert _{\infty}}u_{\Lambda}(x)\leq\rho(x)\leq\left\Vert \rho\right\Vert
_{\infty},\quad\forall\,x\in\overline{\Omega}.
\]
Since $u(x_{\Lambda})=\left\Vert u_{\Lambda}\right\Vert _{\infty}$ we conclude
that $\rho(x_{\Lambda})=\left\Vert \rho\right\Vert _{\infty}.$ Noting that
\[
\frac{\left\Vert u_{\Lambda}\right\Vert _{\infty}}{\left\Vert \rho\right\Vert
_{\infty}}=\left\Vert u_{\Lambda}\right\Vert _{\infty}\Lambda_{\infty}=\left(
\frac{\Lambda_{\infty}}{\Lambda}\right)  ^{\frac{1}{1-Q}}%
\]
we obtain (\ref{uinfrho}), with the equality holding at $x_{\Lambda}$ and also
on $\partial\Omega$ (since $u_{\Lambda}=\rho=0$ on $\partial\Omega$).
\end{proof}

\begin{corollary}
\label{rem}Lemma \eqref{lemapqdelta}, Theorem \eqref{uinf} and Corollary
\eqref{corol} remain true for $\Lambda=\Lambda_{\infty}$ in both cases
$Q\in(0,1)$ and $Q\in(1,\infty)$, if one takes $\lambda_{p}=c\left\vert
\Omega\right\vert (\Lambda_{\infty})^{p},$ with $c>1.$
\end{corollary}

\begin{proof}
It is proved in \cite{EP16} that the function $(N,\infty)\ni m\mapsto\left(
\left\vert \Omega\right\vert ^{-1}\lambda_{\infty}(m)\right)  ^{1/m}$ is
increasing. It follows that
\[
\left(  \left\vert \Omega\right\vert ^{-1}\lambda_{\infty}(p)\right)
^{1/p}\leq\lim_{m\rightarrow\infty}\left(  \left\vert \Omega\right\vert
^{-1}\lambda_{\infty}(m)\right)  ^{1/m}=\Lambda_{\infty}.
\]
Hence, by taking $\lambda_{p}=c\left\vert \Omega\right\vert (\Lambda_{\infty
})^{p}$ with $c>1$ we have $\lim_{p\rightarrow\infty}(\lambda_{p})^{\frac
{1}{p}}=\Lambda_{\infty}$ and
\[
\left(  \left\vert \Omega\right\vert ^{-1}\lambda_{\infty}(p)\right)
^{1/p}\leq\Lambda_{\infty}<c^{1/p}\Lambda_{\infty},
\]
so that $\lambda_{\infty}(p)<\lambda_{p}.$ Proposition \ref{existpqinf} then
guarantees that \eqref{pqdelta} has a nonnegative least energy solution
$u_{p}.$ Following the proofs of Lemma \ref{lemapqdelta}, Theorem \ref{uinf}
and Corollary \ref{corol}, we obtain a nonnegative function $u_{\Lambda
_{\infty}}\in W^{1,\infty}(\Omega)\cap C_{0}(\overline{\Omega})$ as the
uniform limit in $\overline{\Omega}$ of a sequence $\left\{  u_{p_{n}%
}\right\}  ,$ with $p_{n}\rightarrow\infty.$ Moreover, such a function
satisfies
\[
u_{\Lambda_{\infty}}(x_{\Lambda_{\infty}})=\left\Vert u_{\Lambda_{\infty}%
}\right\Vert _{\infty}=\frac{1}{\Lambda_{\infty}},\quad\left\Vert \nabla
u_{\Lambda_{\infty}}\right\Vert _{\infty}=1
\]
and%
\[
0\leq u_{\Lambda_{\infty}}(x)\leq\rho(x)\quad\forall\,x\in\overline{\Omega},
\]
so that $x_{\Lambda_{\infty}}$ is also a maximum point of $\rho.$
\end{proof}

\begin{remark}
Recalling that $\lim_{p\rightarrow\infty}(\lambda_{\infty}(p))^{\frac{1}{p}%
}=\Lambda_{\infty},$ one can see that if $\lambda_{p}$ is such that
$\lim_{p\rightarrow\infty}(\lambda_{p})^{\frac{1}{p}}=\Lambda<\Lambda_{\infty
},$ then $\lambda_{p}<\lambda_{\infty}(p)$ for all $p$ large enough. Thus,
according to Remark \ref{nonex}, if $\Lambda<\Lambda_{\infty}$ the problem
\eqref{pqdelta} has no weak solution for all $p$ large enough.
\end{remark}

Before determining the equation satisfied by $u_{\Lambda},$ let us recall some
definitions. In what follows $D$ denotes a bounded domain of $\mathbb{R}^{N},$
$N\geq2.$ Further up we will take $D=\Omega\setminus\{x_{\Lambda}\}.$

\begin{definition}
\label{def3}Let $u\in C(\overline{D}),$ $\phi\in C^{2}(\Omega)$ and $x_{0}\in
D.$ We say that $\phi$ touches $u$ at $x_{0}$ from below if%
\[
\phi(x)-u(x)<0=\phi(x_{0})-u(x_{0}),\quad\forall\,x\in D\setminus\{x_{0}\}.
\]
Analogously, we say that $\phi$ touches $u$ at $x_{0}$ from above if%
\[
\phi(x)-u(x)>0=\phi(x_{0})-u(x_{0}),\quad\forall\,x\in D\setminus\{x_{0}\}.
\]

\end{definition}

In the sequel we recall the concept of viscosity solution for an equation in
the form
\begin{equation}
F(u,\nabla u,D^{2}u)=0\quad\mathrm{in}\,D. \label{Fu}%
\end{equation}
The differential operator $F(u,\nabla u,D^{2}u)$ includes two operators we are
interested in, which are the $\infty$-Laplacian%
\[
\Delta_{\infty}u:=\frac{1}{2}\nabla u\cdot\nabla\,\left\vert \nabla
u\right\vert ^{2}=\sum_{i,j=1}^{N}u_{x_{i}}u_{x_{j}}u_{x_{i}x_{j}}%
\]
and the $(p,q)$-Laplacian%
\[
\left(  \Delta_{p}+\Delta_{q}\right)  u:=\left(  \left\vert \nabla
u\right\vert ^{p-4}+\left\vert \nabla u\right\vert ^{q-4}\right)  \left\vert
\nabla u\right\vert ^{2}\Delta u+\left(  (p-2)\left\vert \nabla u\right\vert
^{p-4}+(q-2)\left\vert \nabla u\right\vert ^{q-4}\right)  \Delta_{\infty}u,
\]
where $\Delta u=\sum_{i=1}^{N}u_{x_{i}x_{i}}$ is the Laplacian.

\begin{definition}
\label{def4}We say that $u\in C(\overline{D})$ is a viscosity subsolution of
\eqref{Fu} if
\[
F(\phi(x_{0}),\nabla\phi(x_{0}),D^{2}\phi(x_{0}))\geq0
\]
whenever $x_{0}\in D$ and $\phi\in C^{2}(D)$ are such that $\phi$ touches $u$
from above at $x_{0}.$ Analogously, we say that $u$ is a viscosity
supersolution of \eqref{Fu} if
\[
F(\phi(x_{0}),\nabla\phi(x_{0}),D^{2}\phi(x_{0}))\leq0
\]
whenever $x_{0}\in D$ and $\phi\in C^{2}(D)$ are such that $\phi$ touches $u$
from below at $x_{0}.$
\end{definition}

\begin{definition}
Let $u\in C(\overline{D}).$ We say that $u$ is viscosity solution of
\eqref{Fu} if $u$ is both a viscosity subsolution and a viscosity
supersolution of \eqref{Fu}.
\end{definition}

\begin{definition}
We say that $u\in C(\overline{D})$ is $(p,q)$-subharmonic (respectively,
$(p,q)$-superharmonic and $(p,q)$-harmonic) in $D$ if $u$ is a viscosity
subsolution (respectively, supersolution and solution) of
\[
\left(  \Delta_{p}+\Delta_{q}\right)  u=0\quad\mathrm{in}\,D.
\]

\end{definition}

\begin{definition}
We say that $u\in C(\overline{D})$ is $\infty$-subharmonic (respectively,
$\infty$-superharmonic and $\infty$-harmonic) in $D$ if $u$ is a viscosity
subsolution (respectively, supersolution and solution) of
\[
\Delta_{\infty}u=0\quad\mathrm{in}\,D.
\]

\end{definition}

The next lemma is adapted from \cite{Li15}.

\begin{lemma}
\label{pqvisc}Let $N<m<p,q<\infty$ and suppose that $u\in C(D)\cap W_{0}%
^{1,m}(D)$ is a weak solution of
\[
(\Delta_{p}+\Delta_{q})u=0\quad\mathrm{in}\,D,
\]
that is,%
\begin{equation}%
{\displaystyle\int_{D}}
\left(  \left\vert \nabla u\right\vert ^{p-2}+\left\vert \nabla u\right\vert
^{q-2}\right)  \nabla u\cdot\nabla\eta\mathrm{d}x=0,\quad\forall\,\eta\in
C_{0}^{\infty}(D). \label{pqvisca}%
\end{equation}
Then $u$ is $(p,q)$-harmonic in $D.$
\end{lemma}

\begin{proof}
Suppose, by contradiction, that $u$ is not $(p,q)$-superharmonic in $D.$ Then,
there exist $x_{0}\in D$ and $\phi\in C^{2}(D)$ touching $u$ at $x_{0}$ from
below such that $(\Delta_{p}+\Delta_{q})\phi(x_{0})>0.$ By continuity, this
strict inequality holds in ball $B_{2\epsilon}(x_{0})\subset D,$ that is,
\begin{equation}
\left(  \left\vert \nabla\phi\right\vert ^{p-4}+\left\vert \nabla
\phi\right\vert ^{q-4}\right)  \left\vert \nabla\phi\right\vert ^{2}\Delta
\phi+\left(  (p-2)\left\vert \nabla\phi\right\vert ^{p-4}+(q-2)\left\vert
\nabla\phi\right\vert ^{q-4}\right)  \Delta_{\infty}\phi>0\quad\mathrm{in}%
\,B_{2\epsilon}(x_{0}). \label{pqv}%
\end{equation}
Define
\[
\psi(x)=\phi(x)+\frac{\alpha}{2},\quad x\in B_{\epsilon}(x_{0}),
\]
where
\[
\alpha:=\min\left\{  u(x)-\phi(x):x\in\partial B_{\epsilon}(x_{0})\right\}  .
\]
Note that $\alpha>0$ since $u(x)>\phi(x)$ for all $x\in D\setminus\{x_{0}\}.$
Hence, $\psi(x_{0})=u(x_{0})+\alpha/2>u(x_{0})$ and%
\[
\psi(x)=u(x)-(u(x)-\phi(x))+\frac{\alpha}{2}\leq u(x)-\frac{\alpha}%
{2}<u(x)\quad\forall\,x\in\partial B_{\epsilon}(x_{0}).
\]

Let $D_{\epsilon}$ be a subdomain of $B_{\epsilon}(x_{0})$ such that $\psi>u$
in $D_{\epsilon}$ and $\psi=u$ on $\partial D_{\epsilon}.$ In view of
(\ref{pqv}) we have%
\[
\operatorname{div}\left[  \left(  \left\vert \nabla\psi\right\vert
^{p-2}+\left\vert \nabla\psi\right\vert ^{q-2}\right)  \nabla\psi\right]
=\operatorname{div}\left[  \left(  \left\vert \nabla\phi\right\vert
^{p-2}+\left\vert \nabla\phi\right\vert ^{q-2}\right)  \nabla\phi\right]
>0\quad\mathrm{in}\,B_{2\epsilon}(x_{0}),
\]
so that
\[%
{\displaystyle\int_{D_{\epsilon}}}
\left(  \left\vert \nabla\psi\right\vert ^{p-2}+\left\vert \nabla
\psi\right\vert ^{q-2}\right)  \nabla\psi\cdot\nabla\eta\mathrm{d}x\leq
0,\quad\forall\,\eta\in C_{0}^{\infty}(B_{\epsilon}(x_{0})),\quad\eta\geq0.
\]

Combining this inequality with (\ref{pqvisca}) and recalling that
$(\psi-u)_{+}\in W_{0}^{1,m}(B_{\epsilon}(x_{0}))$ can be approximated in
$W_{0}^{1,m}(B_{\epsilon}(x_{0}))$ by functions in $C_{0}^{\infty}%
(B_{\epsilon}(x_{0}))$ we obtain%
\[%
{\displaystyle\int_{B_{\epsilon}(x_{0})}}
\left[  \left(  \left\vert \nabla\psi\right\vert ^{p-2}\nabla\psi-\left\vert
\nabla u\right\vert ^{p-2}\nabla u\right)  +\left(  \left\vert \nabla
\psi\right\vert ^{q-2}\nabla\psi-\left\vert \nabla u\right\vert ^{q-2}\nabla
u\right)  \right]  \cdot\nabla\left(  \psi-u\right)  _{+}\mathrm{d}x\leq0.
\]
Taking (\ref{desvec}) into account, we conclude that $\psi\leq u$ in
$B_{\epsilon}(x_{0})$, which contradicts the fact that $\psi>u$ in a
neighborhood of $x_{0}$ (recall that $\psi(x_{0})>u(x_{0})$).

Analogously, we arrive at a contradiction if we assume that $u$ is not
$(p,q)$-subharmonic in $D.$
\end{proof}

The following lemma is taken from \cite{Li15}.

\begin{lemma}
\label{calc}Suppose that $f_{n}\rightarrow f$ uniformly in $\overline{D},$
$f_{n},\;f\in C(\overline{D}).$ If $\phi\in C^{2}(D)$ touches $f$ from below
at $x_{0},$ then there exists $x_{n_{j}}\rightarrow x_{0}$ such that
\[
f(x_{n_{j}})-\phi(x_{n_{j}})=\min_{D}\left\{  f_{n_{j}}-\phi\right\}  .
\]

\end{lemma}

In the sequel, $u_{\Lambda}$ denotes the function obtained in Theorem
\ref{uinf}, for $\Lambda>\Lambda_{\infty},$ and $u_{\Lambda_{\infty}}$ denotes
the function described in Corollary \ref{rem} (for $\Lambda=\Lambda_{\infty}$).

\begin{theorem}
\label{diric}The function $u_{\Lambda}$ is $\infty$-harmonic in $D=\Omega
\setminus\left\{  x_{\Lambda}\right\}  .$ Therefore, $u_{\Lambda}$ is strictly
positive in $\Omega$ and attains its maximum point only at $x_{\Lambda}.$
\end{theorem}

\begin{proof}
Let $x_{0}\in D$ and take $\phi\in C^{2}(D)$ touching $u_{\Lambda}$ from below
at $x_{0}.$ Thus,
\[
\phi(x)-u_{\Lambda}(x)<0=\phi(x_{0})-u_{\Lambda}(x_{0}),\quad\mathrm{if}%
\,x\not =x_{0}.
\]

If $\left\vert \nabla\phi(x_{0})\right\vert =0$ then we trivially have
\[
\Delta_{\infty}\phi(x_{0})=\sum_{i,j=1}^{N}\frac{\partial\phi}{\partial x_{i}%
}(x_{0})\frac{\partial\phi}{\partial x_{j}}(x_{0})\frac{\partial^{2}\phi
}{\partial x_{i}\partial x_{j}}(x_{0})=0.
\]

So, we assume that $\left\vert \nabla\phi(x_{0})\right\vert \not =0.$ Let
$B_{\epsilon}(x_{0})\subset D$ be a ball centered at $x_{0}$ with radius
$\epsilon>0$ such that $\left\vert \nabla\phi\right\vert >0$ in $B_{\epsilon
}(x_{0}).$

Let $u_{n},$ $p_{n}$ and $x_{p_{n}}$ given in Theorem \ref{uinf}. Since
$x_{p_{n}}\rightarrow x_{\Lambda}\not =x_{0}$ we can take $n_{0}>\mathbb{N}$
such that $x_{p_{n}}\not \in B_{\epsilon}(x_{0})$ for all $n>n_{0}.$
Consequently,
\begin{equation}%
{\displaystyle\int_{B_{\epsilon}(x_{0})}}
\left(  \left\vert \nabla u_{p_{n}}\right\vert ^{p_{n}-2}+\left\vert \nabla
u_{p_{n}}\right\vert ^{q(p_{n})-2}\right)  \nabla u_{p_{n}}\cdot\nabla
\varphi\mathrm{d}x=0,\quad\forall\,\varphi\in C_{0}^{\infty}(B_{\epsilon
}(x_{0}))\,\mathrm{and}\,n\geq n_{0}. \label{pqhar}%
\end{equation}
We recall that $u_{p_{n}}\in W_{0}^{1,m}(\Omega)$ for all $n$ sufficiently
large, where $m>N$ is fixed. Thus, combining (\ref{pqhar}) and Lemma
\ref{pqvisc} we conclude that $u_{p_{n}}$ is a viscosity solution of
\[
\left(  \Delta_{p_{n}}+\Delta_{q(p_{n})}\right)  u=0\quad\mathrm{in}%
\,B_{\epsilon}(x_{0}),\quad\forall\,n\geq n_{0}.
\]

Applying Lemma \ref{calc} we can take $\left\{  x_{n_{j}}\right\}  \subset
B_{\epsilon}(x_{0})$ such that $x_{n_{j}}\rightarrow x_{0}$ and
\[
\alpha_{j}:=\min_{B_{\epsilon}(x_{0})}\left\{  u_{p_{n_{j}}}-\phi\right\}
=u_{\Lambda}(x_{n_{j}})-\phi(x_{n_{j}})\leq u_{p_{n_{j}}}(x)-\phi(x),\quad
x\not =x_{n_{j}}.
\]

The function $\psi(x):=\phi(x)+\alpha_{j}-\left\vert x-x_{n_{j}}\right\vert
^{4}$ belongs to $C^{2}(B_{\epsilon}(x_{0}))$ and
\begin{align*}
\psi(x)-u_{p_{n_{j}}}(x)  &  =\phi(x)-u_{p_{n_{j}}}(x)+\alpha_{j}-\left\vert
x-x_{n_{j}}\right\vert \\
&  \leq-\left\vert x-x_{n_{j}}\right\vert ^{4}<0=\psi(x_{n_{j}})-u_{p_{n_{j}}%
}(x_{n_{j}}),\quad x\not =x_{n_{j}}.
\end{align*}
That is, $\psi$ touches $u_{n_{j}}$ from below at $x_{n_{j}}.$ It follows
that
\[
(\Delta_{p_{n_{j}}}+\Delta_{q(p_{n_{j}})})\psi(x_{n_{j}})\leq0.
\]

Since $\left\vert \nabla\psi(x_{n_{j}})\right\vert =\left\vert \nabla
\phi(x_{n_{j}})\right\vert >0$ and%
\begin{align*}
(\Delta_{p_{n_{j}}}+\Delta_{q(p_{n_{j}})})\psi(x_{n_{j}})  &  =\left(
\left\vert \nabla\psi(x_{n_{j}})\right\vert ^{p_{n_{j}}-4}+\left\vert
\nabla\psi(x_{n_{j}})\right\vert ^{q(p_{n_{j}})-4}\right)  \left\vert
\nabla\psi(x_{n_{j}})\right\vert ^{2}\Delta\psi(x_{n_{j}})\\
&  +\left(  (p_{n_{j}}-2)\left\vert \nabla\psi(x_{n_{j}})\right\vert
^{p_{n_{j}}-4}+(q(p_{n_{j}})-2)\left\vert \nabla\psi(x_{n_{j}})\right\vert
^{q(p_{n_{j}})-4}\right)  \Delta_{\infty}\psi(x_{n_{j}})
\end{align*}
we obtain%
\begin{equation}
\Delta_{\infty}\psi(x_{n_{j}})\leq-\frac{\left(  \left\vert \nabla
\psi(x_{n_{j}})\right\vert ^{p_{n_{j}}-4}+\left\vert \nabla\psi(x_{n_{j}%
})\right\vert ^{q(p_{n_{j}})-4}\right)  \left\vert \nabla\psi(x_{n_{j}%
})\right\vert ^{2}\Delta\psi(x_{n_{j}})}{(p_{n_{j}}-2)\left\vert \nabla
\psi(x_{n_{j}})\right\vert ^{p_{n_{j}}-4}+(q(p_{n_{j}})-2)\left\vert
\nabla\psi(x_{n_{j}})\right\vert ^{q(p_{n_{j}})-4}}. \label{super}%
\end{equation}

Noting that
\[
\lim_{j\rightarrow\infty}\left\vert \nabla\psi(x_{n_{j}})\right\vert
^{2}\Delta\psi(x_{n_{j}})=\lim_{j\rightarrow\infty}\left\vert \nabla
\phi(x_{n_{j}})\right\vert ^{2}\Delta\phi(x_{n_{j}})=\left\vert \nabla
\phi(x_{0})\right\vert ^{2}\Delta\phi(x_{0})
\]
and
\[
0\leq\frac{\left(  \left\vert \nabla\psi(x_{n_{j}})\right\vert ^{p_{n_{j}}%
-4}+\left\vert \nabla\psi(x_{n_{j}})\right\vert ^{q(p_{n_{j}})-4}\right)
}{(p_{n_{j}}-2)\left\vert \nabla\psi(x_{n_{j}})\right\vert ^{p_{n_{j}}%
-4}+(q(p_{n_{j}})-2)\left\vert \nabla\psi(x_{n_{j}})\right\vert ^{q(p_{n_{j}%
})-4}}\leq\max\left\{  \frac{1}{p_{n_{j}}-2},\frac{1}{q(p_{n_{j}})-2}\right\}
\]
we can see that the right-hand side of (\ref{super}) tends to zero as
$j\rightarrow\infty.$ Therefore, letting $j\rightarrow\infty$ in (\ref{super})
we arrive at%
\[
\Delta_{\infty}\phi(x_{0})=\lim_{j\rightarrow\infty}\Delta_{\infty}%
\psi(x_{n_{j}})\leq0,
\]
concluding thus that $u_{\Lambda}$ is $\infty$-superharmonic in $D.$

Analogously, we can prove that $u_{\Lambda}$ is also $\infty$-subharmonic in
$D.$

As in \cite{EP16} we can apply the Harnack inequality (see \cite{LqManf}) and
the comparison principle (see \cite{BB, CEG, Jensen}), both for $\infty
$-harmonic functions, to prove, respectively, that $u_{\Lambda}$ is strictly
positive in $\Omega$ and that its maximum point is attained only at
$x_{\Lambda}.$ The comparison principle is used to compare $u_{\Lambda}$ with
the function $v(x):=\left\Vert u_{\Lambda}\right\Vert _{\infty}\left(
1-\frac{1}{\beta}\left\vert x-x_{\Lambda}\right\vert \right)  ,$ where
$\beta=\max\left\{  \left\vert x-x_{\Lambda}\right\vert :x\in\partial
\Omega\right\}  .$ This function is $\infty$-harmonic in $D=\Omega
\setminus\left\{  x_{\Lambda}\right\}  $ and such that $v\geq u_{\Lambda}$ on
$\partial D=\partial\Omega\cup\left\{  x_{\Lambda}\right\}  .$ Hence,
\[
u_{\Lambda}(x)\leq v(x)=\left\Vert u_{\Lambda}\right\Vert _{\infty}\left(
1-\frac{1}{\beta}\left\vert x-x_{\Lambda}\right\vert \right)  <\left\Vert
u_{\Lambda}\right\Vert _{\infty},\quad\forall\,x\in D.
\]

\end{proof}

The following result applies when $\Omega$ is a ball, a square and many other
symmetric domains, even nonconvex ones.

\begin{corollary}
Suppose that $\Omega$ is such that the distance function to its boundary has a
unique maximum point $x_{\rho}.$ If $\Lambda>\Lambda_{\infty}$ and
$Q\in(0,1),$ then
\[
u_{\Lambda}=\left(  \Lambda_{\infty}/\Lambda\right)  ^{\frac{1}{1-Q}%
}u_{\Lambda_{\infty}}.
\]

\end{corollary}

\begin{proof}
Let $v:=\left(  \Lambda_{\infty}/\Lambda\right)  ^{1/(1-Q)}u_{\Lambda_{\infty
}}$ where $u_{\Lambda_{\infty}}$ is the function described in the Corollary
\ref{rem}. Taking into account Corollaries \ref{corol} and \ref{rem} we have
$x_{\Lambda}=x_{\rho}$ and
\[
v(x_{\rho})=\left\Vert v\right\Vert _{\infty}=\left(  \Lambda_{\infty}%
/\Lambda\right)  ^{\frac{1}{1-Q}}\left\Vert u_{\Lambda_{\infty}}\right\Vert
_{\infty}=\left(  \Lambda_{\infty}/\Lambda\right)  ^{\frac{1}{1-Q}}%
(1/\Lambda_{\infty})=u_{\Lambda}(x_{\rho}),\quad\Lambda\geq\Lambda_{\infty}.
\]
It follows that both $v$ and $u_{\Lambda}$ are functions in $C(\overline
{\Omega})$ that solve, in the viscosity sense, the problem
\[
\left\{
\begin{array}
[c]{lll}%
\Delta_{\infty}u=0 & \mathrm{in} & \Omega\setminus\left\{  x_{\rho}\right\} \\
u=0 & \mathrm{on} & \partial\Omega\\
u(x_{\rho})=\left(  \Lambda_{\infty}/\Lambda\right)  ^{1/(1-Q)}(1/\Lambda
_{\infty}). &  &
\end{array}
\right.
\]
Therefore, by uniqueness (see \cite{BB, CEG, Jensen}) we have $v\equiv
u_{\Lambda}.$
\end{proof}

\section{Acknowledgements}

C.O. Alves was partially supported by CNPq/Brazil (304036/2013-7) and
INCT-MAT. G. Ercole was partially supported by CNPq/Brazil (483970/2013-1 and
306590/2014-0) and Fapemig/Brazil (APQ-03372-16).

\end{document}